\documentclass[12pt,notitlepage]{amsart}
\usepackage{latexsym,amsfonts,amssymb,amsmath,amsthm, datetime2}
\usepackage{graphicx}
\usepackage{color}
\usepackage{mathrsfs}
\usepackage{amsmath, amsthm, amsfonts, amssymb}
\usepackage{wrapfig}
\usepackage{stackrel}
\usepackage{color}
\usepackage{float}
\usepackage{enumitem}
\usepackage{mathtools}
\usepackage{multicol}
\usepackage[normalem]{ulem}
\usepackage{xcolor}
\usepackage{hyperref}
\usepackage{cancel}

\usepackage[margin=1in]{geometry}
\usepackage[backend=biber, doi=false,isbn=false,url=false, articlein=false, maxnames=4, maxalphanames=4, sorting=nty, style=ext-alphabetic]{biblatex}
\DeclareFieldFormat[article,misc,inbook,incollection]{title}{#1}
\DeclareLabelalphaTemplate{
  \labelelement{
    \field[final]{shorthand}
    \field{label}
    \field[strwidth=3,strside=left,ifnames=1]{labelname}
    \field[strwidth=1,strside=left]{labelname}
  }
  \labelelement{
    \field[strwidth=2,strside=right]{origyear}
    \field[strwidth=2,strside=right]{year}
  }
}
\newcommand{\mz}{\ensuremath{\mathbb Z}}
\newcommand{\mr}{\ensuremath{\mathbb R}}

\newcommand{\mymod}{\ensuremath{\negthickspace \negmedspace \pmod}}
\newcommand{\shortmod}{\ensuremath{\negthickspace \negthickspace \negthickspace \pmod}}

\newcommand{\intR}{\int_{-\infty}^{\infty}}

\newcommand{\sumstar}{\sideset{}{^*}\sum}

\DeclarePairedDelimiter{\abs}{\lvert}{\rvert}

\newcommand{\cond}{\mathrm{cond}}

\theoremstyle{plain}		
	\newtheorem{mytheo}{Theorem}[section]
	
	\newtheorem{myprop}[mytheo]{Proposition}
	\newtheorem{mycoro}[mytheo]{Corollary}
     \newtheorem{mylemma}[mytheo]{Lemma}
	\newtheorem{mydefi}[mytheo]{Definition}
	
	\newtheorem{myremark}[mytheo]{Remark}

\theoremstyle{remark}

\numberwithin{equation}{section}
\begin{document}
\author{Agniva Dasgupta}
\email{agniva@tamu.edu}
\author{Wing Hong Leung}
\email{josephleung@tamu.edu}
 \author{Matthew P. Young} 
  \email{myoung@math.tamu.edu}
 \address{Department of Mathematics \\
 	  Texas A\&M University \\
 	  College Station \\
 	  TX 77843-3368}
  \thanks{This material is based upon work supported by the National Science Foundation under agreement No. DMS-2302210 (M.Y.).  Any opinions, findings and conclusions or recommendations expressed in this material are those of the authors and do not necessarily reflect the views of the National Science Foundation.}

% \begin{abstract}
% \end{abstract}
\title{The second moment of the $GL_3$ standard $L$-function on the critical line
%\author{Delta Seminar
%\\ Last compiled: \DTMnow
}

\begin{abstract}
 We obtain a strong bound on the second moment of the $GL_3$ standard $L$-function on the critical line.  The method builds on the recent work of Aggarwal, Leung, and Munshi which treated shorter intervals.  We deduce some corollaries including an improvement on the error term in the Rankin-Selberg problem, and on certain subconvexity bounds for $GL_3 \times GL_2$ and $GL_3$ $L$-functions.
 As a byproduct of the method of proof, we also obtain an estimate for an average of shifted convolution sums of $GL_3$ Fourier coefficients.
 \end{abstract}

\maketitle

\section{Introduction}
\subsection{Statement of results} \label{subsec:statement}
Let $f$ be a Hecke cusp form for $SL_3(\mz)$, and let $L(f,s)$ denote the $L$-function associated to it. The main purpose of this paper is to establish the following estimate for the second moment.
\begin{mytheo}
\label{thm:mainthm}
 Let $T \geq 1$.  Then for any $\varepsilon>0$,
 \begin{equation}
 \label{eq:mainthm}
  \int_{-T}^{T} |L(f, 1/2+it)|^2 dt \ll_{f,\varepsilon} T^{4/3+\varepsilon}.
 \end{equation}
\end{mytheo}

To help gauge the strength of this result, note that the analytic 
conductor of $L(f, 1/2+it)$ is $q_{\infty}(t) = (1+|t|)^3$. The
mean value theorem for Dirichlet polynomials would thus replace the exponent $4/3$ in \eqref{eq:mainthm} by $3/2$, thereby indicating the result in Theorem \ref{thm:mainthm} is nontrivial.  
In particular, from  \eqref{eq:mainthm} we deduce the subconvexity bound $L(f,1/2+it) \ll q_{\infty}(t)^{2/9 + \varepsilon}$. 
%(note that $\tfrac29 = \tfrac14 - \tfrac{1}{36}$).

Formally taking $f$ to be an Eisenstein series in \eqref{eq:mainthm}, one would also obtain the bound $\int_{0}^{T} |\zeta(1/2+it)|^6 dt \ll T^{a+\varepsilon}$, with $a=4/3$.  However, 
in this case, one can deduce $a=5/4$ using H\"{o}lder's inequality to interpolate the twelfth moment bound
%$\int_0^{T} |\zeta(1/2+it)|^{12} dt \ll T^{2+\varepsilon}$ 
of Heath-Brown 
\cite{HB} and the (easy) fourth moment bound.
%$\int_0^{T} |\zeta(1/2+it)|^4 dt \ll T^{1+\varepsilon}$. 
One should note however that the techniques for estimating the twelfth moment of zeta (as in \cite{HB} and \cite{IwaniecZeta}) do not generalize to $GL_3$ cusp forms.

Our method builds on and extends the important recent work of Aggarwal, Leung, and Munshi, who showed \cite[Thm $1.1$]{ALM} that for any $\varepsilon>0$
\begin{equation}\label{eq:ALM}
    \int_{T}^{T+T^{2/3}}  \Big|L(f, 1/2 + it ) \Big|^2 dt \ll_{f,\varepsilon} T^{5/4+\varepsilon}.  
\end{equation}  
In Section \ref{subsec:Outline} below we discuss in greater detail the interaction between our Theorem \ref{thm:mainthm} and this result; we also provide a sketch of the important steps that go into the proof of Theorem \ref{thm:mainthm}. 
Note that \eqref{eq:ALM} gives a stronger subconvexity bound for $L(f,s)$, namely
\begin{equation}\label{eq:ALMsubcon}
    L(f,1/2+it) \ll_{\varepsilon} q_{\infty}(t)^{5/24 + \varepsilon}.
\end{equation}
  The bound in \eqref{eq:ALMsubcon} is currently the strongest subconvexity result known for general $GL_3$ cusp forms, though see \eqref{eq:subconvexityselfdual} below for the self-dual case.
  %(note again that $\tfrac{5}{24} = \tfrac29 - \tfrac{1}{72}$). 

Theorem \ref{thm:mainthm} also improves on a recent result of Pal \cite[Theorem 1]{Pal}, which states
\begin{equation} \label{eq:Pal2mom}
    \int_{T}^{2T} |L(f, 1/2+it)|^2 dt \ll_{f,\varepsilon} T^{\frac32 - \frac{3}{88}+\varepsilon}.
\end{equation}

Our method proceeds by estimating weighted sums of the form
$\sum_{h,n} \lambda_f(n) \lambda_f(n+h)$.  
Such averages have been considered in \cite{BBMZ} with $\lambda_f(n)$ replaced by $d_3(n)$, which is of relevance for the sixth moment of the zeta function.  
The method of \cite{BBMZ} splits $d_3$ as $1*1*1$ which is a technique unavailable for $\lambda_f$.  We refer to \cite{BBMZ} for precise statements of their results, and only mention here that the quality of their estimate on the shifted divisor problem is able to imply $\int_0^{T} |\zeta(1/2+it)|^6 dt \ll T^{11/8+\varepsilon}$.  Note that $4/3 < 11/8$, so Theorem \ref{thm:mainthm} appears favorable even in the case of the zeta function.
See also \cite[Theorem 1.3(ii)]{MRT} for some improvements on \cite{BBMZ}.
Our result on the shifted convolution sum is as follows:
\begin{mytheo}\label{thm:shiftedSum}
    Let $1 \leq H \leq N^{1/2-\varepsilon}$. Let $W$ be an $X^\varepsilon$-inert function with fixed support on $(\mathbb{R}^+)^2$. Then \begin{align*}
        \sum_{n,k}\lambda(n)\lambda_f(n+k)W\left(\frac{n}{N},\frac{k}{H}\right)\ll_{f,\varepsilon} \frac{N^{4/3+\varepsilon}}{H^{1/3}}+\sqrt{H}N^{1+\varepsilon}.
    \end{align*}
\end{mytheo}
Note that this bound is nontrivial as long as $H>N^{1/4}$. Moreover, the restriction of $H^2\leq N^{1-\varepsilon}$ is only made for the ease of demonstration. In fact, the problem becomes significantly easier when $H^2\geq N$.
%This may be of independent interest to the reader, and 
We present the proof of Theorem \ref{thm:shiftedSum} in Section \ref{sect:ShiftedSum}.

The long interval bound in \eqref{eq:mainthm} has some nice corollaries that we discuss now.
\begin{mycoro}[Rankin-Selberg Problem]
\label{coro:RankinSelberg}
Let $g$ be a fixed Hecke-Maass cusp form for $SL_2(\mz)$.  
Let $c_g = \frac{6}{\pi^2} L(\mathrm{sym}^2 g, 1)$.
Then for any $\varepsilon >0$ we have
\begin{equation}\label{eq:RScoro}
    \sum_{n \leq x} |\lambda_g(n)|^2 = c_g x + O_{g, \varepsilon}(x^{4/7+\varepsilon}).
\end{equation}
\end{mycoro}
The Rankin-Selberg problem aims to reduce the exponent of $x$ in the error term present in \eqref{eq:RScoro}. The classical result, due independently to Rankin (\cite{Ra}) and Selberg (\cite{Sel}), established an error term of size $O_g(x^{\frac35+\varepsilon})$. 
In a recent breakthrough (\cite{Huang1}), Huang gave the first improvement over this long-standing exponent $3/5$, by obtaining the exponent $3/5 - \delta + o(1)$, with $\delta = \frac{1}{560}$. This has been improved successively by Pal (\cite[Cor. 1.3]{Pal}) with $\delta=\frac{6}{1085}$, and then by Huang again (\cite{Huang2}) with $\delta=\frac{3}{305}$. We note that \eqref{eq:RScoro} improves these results to $\delta = \frac{1}{35}$.
The application of Theorem \ref{thm:mainthm} to the Rankin-Selberg problem is fairly straightforward, and uses $f = \mathrm{sym}^2 g$.

\begin{mycoro}
\label{coro:LNQ}
 Let $f = \mathrm{sym}^2 g$, where $g$ is a fixed $SL_2(\mz)$ Hecke-Maass cusp form. For $\varepsilon >0$, $T$ large, and with $1 \leq \Delta \leq T$, we have 
 \begin{equation}
 \label{eq:Xiaoqing}
     \sum_{T \leq t_j \leq T + \Delta} L(1/2, f \otimes u_j)
     + \int_{T}^{T+\Delta} |L(1/2+it, f)|^2 dt \ll_{f, \varepsilon}  
     T^{1+\varepsilon}(\Delta + T^{1/7}),
     %\Delta T^{1+\varepsilon} + \frac{T^{7/6+\varepsilon}}{\Delta^{1/6}},
 \end{equation}
 where the sum is over Hecke-Maass cusp forms $u_j$ (with spectral parameter $t_j$) on $SL_2(\mz)$.
In particular, we have the Lindel\"of-on-average bound for $\Delta \geq T^{1/7}$.  We also deduce the strong subconvexity bounds
\begin{equation}
\label{eq:subconvexityselfdual}
    L(1/2, f \otimes u_j) \ll_{f, \varepsilon} (1+|t_j|)^{8/7 + \varepsilon},
    \qquad
    L(1/2 + it, f) \ll_{f, \varepsilon} (1+|t|)^{4/7+\varepsilon}.
\end{equation}
\end{mycoro}
This corollary follows by simply applying Theorem \ref{thm:mainthm} into the work of Lin, Nunes, and Qi \cite{LNQ}.  The moment problem in \eqref{eq:Xiaoqing} was first introduced by Xiaoqing Li in \cite{Li}, and the method of \cite{LNQ} eventually reduces to an estimate for the second moment estimated in \eqref{eq:mainthm}.  The best bound available in \cite{LNQ} was the exponent $3/2$ from the mean value theorem for Dirichlet polynomials. 
Applying our new bound as in Theorem \ref{thm:mainthm} then leads to Corollary \ref{coro:LNQ}. Pal \cite[Cor. 1.2]{Pal} also made an improvement along similar lines, using the bound in \eqref{eq:Pal2mom}.
Curiously, the required bound as in \cite[eq. (6)]{LNQ} has the same structure as in the Rankin-Selberg problem; for more details, see \cite[\S 3.2, \S3.3]{Pal}.  See also \cite[\S 3]{HK} for an alternative proof of the subconvexity results in \cite{LNQ}.

Another type of application of Theorem \ref{thm:mainthm} is to study sums of Dirichlet series coefficients multiplied by an oscillatory factor.  As a sample result in this direction, we state the following.
\begin{mycoro}
\label{coro:nonlinear}
    Let $g$ be a fixed $SL_2(\mz)$ Hecke-Maass cusp form.  
    Suppose that $w$ is a fixed smooth function supported on $[1,2]$.
    Suppose $P \gg N^{\varepsilon}$ and $\beta \in \mr$, $\beta \neq 0, 1$. Then
    \begin{equation}
        \sum_{n=1}^{\infty} \lambda_g(n)^2 e(P (n/N)^{\beta}) w(n/N) \ll_{\beta, g} N^{1/2+\varepsilon} P^{2/3+\varepsilon},
    \end{equation}
    which is nontrivial for $P \ll N^{3/4-\delta}$.
\end{mycoro}
The proof uses that $L(g \otimes g, s)$ factors as $L(\mathrm{sym}^2 g, s) \zeta(s)$.
The same method can be used to derive results for the coefficients of the product of $GL_3$ and $GL_d$ $L$-functions for $d=1,2,3$.
There are a large number of results on nonlinear twists of Fourier coefficients, 
and this is not the main topic of this paper,
so we merely refer to \cite{KMS, KL} for a more comprehensive list of references.

The proofs that the above corollaries follow from Theorem \ref{thm:mainthm} are either standard or explicitly pointed out by previous works. For completeness, we give brief sketches of the proofs in Section \ref{sect:CorProofSketch}.

\subsection{Notations}  We use standard conventions present in analytic number theory. We use $\varepsilon$ to denote an arbitrarily small positive constant. For brevity of notation, we allow $\varepsilon$ to change depending on the context. 
The expression $F \ll G$ implies there exists some constant $k$ for which $\abs{F} \leq k\cdot G$ for all the relevant $F$ and $G$. We use $F \ll_{\varepsilon} G$ to emphasize that the implied constant $k$ depends on $\varepsilon$ (it may also depend on other parameters). For error terms, we often use the big $O$ notation, so $f(x) = O(g(x))$ implies that $f(x) \ll g(x)$ for sufficiently large $x$. 
As usual, $e(x) = e^{2\pi i x}$, and $e_p(x) = e(\frac{x}{p}) = e^{2\pi i \frac{x}{p}}$. Also, $S(m,n,c) \coloneqq \sideset{}{^*}\sum_{x \mymod{c}} e_c\left ({mx + n\overline{x}}\right )$ denotes the Kloosterman sum.

\subsection{Outline of the Proof}
\label{subsec:Outline}
As mentioned earlier, our method builds on that of \cite{ALM}, who studied the shorter interval moment problem of the form
\begin{equation}
    \int_T^{T+\Delta} |L(f, 1/2+it)|^2 dt.
\end{equation}
Our initial analysis follows theirs: we apply an approximate functional equation, square out the resulting sum, and perform the $t$-integration.  This leads to a sum of the form
\begin{equation}
\label{eq:shiftedsum}
 \frac{T}{N} \sum_{k \asymp  \frac{N}{T}}
    \sum_{n \asymp N}
    \lambda_f(n) \overline{\lambda_f(n+k)},
\end{equation}
with $N \asymp T^{3/2}$. 
Next we apply a variant of the DFI delta symbol (see \cite[Lemma 3.1]{LeungShiftedSum}) where we choose arithmetic conductors of size $c \leq Q$ with $Q^2 = o(N)$.  This reduces the arithmetic conductor at the price of increasing the archimedean conductors present in the delta symbol.  
Here the archimedean conductor in the delta symbol has size $\frac{N}{Q^2}$.
\footnote{Due to the length of the $k$-sum being significantly shorter than the $n$-sum, a certain amount of increase in analytic oscillation does not affect the $k$-sum. Thus, decreasing $Q$ within a certain range decreases the dual length of $k$ at the cost of increasing the dual length of the other variables. This allows a re-balance of the mass of the variables which we optimize to get the final bound.}
We then apply the $GL_3$ Voronoi summation formula (in each variable) as well as Poisson summation in the $k$ variable.  So far these moves more or less coincide with \cite{ALM}, and we arrive at an expression of the form
\begin{equation}
\label{eq:sketchadditive}
  \frac{1}{N} 
    \int_{v \asymp 1}
   \sum_{c \asymp Q}
   \sum_{k \asymp K}
   \Big|\sum_{n \asymp N'} \lambda_f(n) S(\overline{k}, n;c) e \Big(3 n^{1/3} \frac{T^{1/2}}{Q} v \Big) \Big|^2  dv,
\end{equation}
with $K = \frac{QT}{N}$,
and $N' = \frac{N^2}{Q^3}$. 
See Proposition \ref{prop:SafterVoronoi} below for the rigorous statement.

It is also convenient to change the archimedean oscillatory factor into multiplicative characters, using Mellin inversion.  This essentially translates \eqref{eq:sketchadditive} into
\begin{equation*}
    \frac{1}{NY}
    \int_{y \asymp Y}
   \sum_{c \asymp Q}
   \sum_{k \asymp K}
   \Big|\sum_{n \asymp N'} \lambda_f(n) S(\overline{k}, n;c) n^{iy} \Big|^2  dy,
\end{equation*}
where $Y = N/Q^2$, which has been the size of the archimedean phase since the introduction of the delta symbol.  The purpose of this change of basis is to facilitate an eventual usage of the large sieve inequality.  At this point, we give up using any properties of the coefficients $\lambda_f(n)$ and reduce the problem to a norm bound on the following bilinear form:
\begin{equation}
\label{eq:Norm}
    \mathcal{N}(N', Q, k, Y)
    = \max_{\| \alpha \| = 1} \int_{y \asymp Y} \sum_{c \asymp Q} \Big| \sum_{n \asymp N'} \alpha_{n} S(\overline{k}, n;c) n^{iy} \Big|^2 dy.
\end{equation}
In terms of this norm, our bound on the second moment is then essentially
\begin{equation}
\label{eq:maxnormSketch}
    T^{-3/2} Y^{-1} K N' \max_{k \asymp K} \mathcal{N}(N', Q, k, Y).
\end{equation}
See Proposition \ref{prop:secondmomentVSnorm} for the rigorous statement.

To help judge the progress so far, we note that if we hypothesize the most optimistic bound
$\mathcal{N}(N', Q, k, Y) \ll Q (N' + QY)$, then we would recover a bound $T^{\varepsilon}(T + Q^{-2} T^{5/2})$ for the second moment.  With $Q = \sqrt{N} = T^{3/4}$, this would imply the optimal bound on the second moment in Theorem \ref{thm:mainthm}.
Needless to say, it appears difficult to establish such a strong bound on the norm.
Also, it is easy to deduce a bound on this norm using the hybrid large sieve inequality.
However, this approach is only strong enough to recover the trivial bound on our moment problem; see Section 
\ref{section:largesieveinterlude} for more details.

As in \cite{ALM},
we will appeal to the duality principle.
We have the following dual definition for the norm in \eqref{eq:Norm},
\begin{equation}\label{eq:dualnorm}
    \mathcal{N}(N', Q, k, Y)
    = \max_{\| \beta \| = 1}  \sum_{n \asymp N'} 
    \Big|
    \int_{y \asymp Y} \sum_{c \asymp Q} \beta(c,y) S(\overline{k}, n;c) n^{iy} dy \Big|^2.
\end{equation}
Next we open the square in \eqref{eq:dualnorm} and apply Poisson summation in $n$. The conductor here is of size at most $Y Q^2$, so the dual sum has length $\frac{Y Q^2}{N'} = \frac{N}{N'} =\frac{Q^3}{N}$, which is relatively small. The contribution of the zero frequency gives rise to a diagonal term of size $N' Q$.  Among the non-zero frequencies, for the purpose of this sketch, we consider the opposite extreme scenario, where  $(c_1, c_2) = 1$. In this case the relevant finite Fourier transform evaluates as
\begin{equation*}
    e_{c_1}(-c_2 \overline{nk}) e_{c_2}(-c_1 \overline{nk}).
\end{equation*}
Using the standard reciprocity, this converts to 
\begin{equation}
\label{eq:exponentialSketch}
    e_{nk}(c_2 \overline{c_1} + c_1 \overline{c_2}),
\end{equation}
which, by virtue of the sizes of $n,k,c_1$ and $c_2$, substantially reduces the modulus
(note $\frac{nk}{c_1 c_2} \approx \frac{Q^2 T}{N^2} \ll \frac{T}{N} \approx T^{-1/2}$).
The corresponding archimedean integral may be
analysed via stationary phase methods, leading to (see 
\eqref{eq:IofUapproximation}) an expression with a phase roughly of the form
\begin{equation}
\label{eq:phaseIntro}
    \Big(\frac{c_1 c_2 (y_1 -y_2)}{n}\Big)^{i(y_1 - y_2)}.
\end{equation}
The variables $c_1$ and $c_2$ are separated in \eqref{eq:phaseIntro}, but not in \eqref{eq:exponentialSketch}.  However, note that \eqref{eq:exponentialSketch} is a function of $c_1/c_2$ only.  Hence we have the (multiplicative) Fourier decomposition 
\begin{equation}
\label{eq:padicsep}
    e_{nk}(c_2 \overline{c_1} + c_1 \overline{c_2})
    =
    \sum_{\chi \shortmod{nk}} \widehat{G}(\chi) \chi(c_1) \overline{\chi}(c_2),
\end{equation}
where
\begin{equation} 
\label{eq:fouriersketch}
    \widehat{G}(\chi) = \frac{1}{\varphi(nk)} \sum_{t \shortmod{nk}} e_{nk}(t + \overline{t}) \overline{\chi}(t) = \frac{S_{\chi}(1,1;nk)}{\varphi(nk)}.
\end{equation}
Here $S_{\chi}(m,n;c)$ is the twisted Kloosterman sum.  Applying these formulas to \eqref{eq:dualnorm}, we obtain the following sort of contribution towards $\mathcal{N}$: 
\begin{multline*}
    \max_{\| \beta \| = 1}
    \Big|
    \int_{y_1, y_2} 
    \sum_{n}
    \sum_{\chi \shortmod{nk}} \widehat{G}(\chi) 
    \\
    \times
    \sum_{c_1, c_2} \beta(c_1, y_1) \overline{\beta}(c_2, y_2) 
 \Big(\frac{c_1 c_2(y_1 - y_2)}{n}\Big)^{i(y_1 - y_2)} \chi(c_1) \overline{\chi}(c_2)
 d y_1 dy_2 \Big|.
\end{multline*}
By a simple Cauchy-Schwarz argument, this reduces to estimating an expression of the type
\begin{equation*}
    \sum_{n} \sum_{\chi \shortmod{nk}} |\widehat{G}(\chi)| 
    \int_{y, v} \Big| \sum_{c} \beta(c,y) c^{iv} \chi(c) \Big|^2 dy dv.
\end{equation*}
Unfortunately the clean bound $|\widehat{G}(\chi)| \ll (nk)^{-1/2+\varepsilon}$ is not always valid, though the truth is not too far off.  If $nk$ is prime, for example, then this Weil-type bound indeed holds.  For higher prime powers $p^r$ ($r\geq 2$), 
the sharpest bound is more complicated to state (see \cite[Lemma 7.1.1]{Kroesche}), but 
one always saves at least a factor of $\sqrt{p}$.
The sparsity of higher prime powers makes up for the lack of individual savings, and we obtain the same bound as if the Weil-type bound were always true.

The final step is an application of the large sieve inequality which 
leads to (roughly-- see \eqref{eq:secondmomentAfterNormBound} for the rigorous statement)
\begin{equation*}
    \int_{T}^{2T} |L(f, 1/2+it)|^2 dt 
    \ll T^{\varepsilon} \Big(\frac{NT}{Q^2} + \frac{Q T^{3/2}}{\sqrt{N}} \Big),
\end{equation*}
which is optimized at $Q= N^{1/2} T^{-1/6}$, thereby proving Theorem \ref{thm:mainthm}.

The authors also investigated another arrangement which keeps the sum over $k$ in the definition of the norm (rather than taking the max over $k \asymp K$ as in \eqref{eq:maxnormSketch}).  Compared to the above sketch, this setup makes the diagonal term smaller and the off-diagonals larger.  Curiously, it leads to a bound that is optimized with a different value of $Q$, but culminates with the same bound as in Theorem \ref{thm:mainthm}.  This alternative approach is more difficult from an arithmetic perspective, since it leads to more complicated character sum evaluations, roughly of the form
\begin{equation*}
    e_{c_1}(-c_2 \overline{n k_1}) e_{c_2}(-c_1 \overline{nk_2})
     \approx e_{nk_1}(c_2 \overline{c_1}) e_{nk_2}(c_1 \overline{c_2}).
\end{equation*}
The difficulty here is that we have two different moduli $nk_1$ and $nk_2$.
Another technical (but not essential) issue is that to use the varying-modulus large sieve one needs primitive characters, 
while in \eqref{eq:padicsep} the sum is over all characters.
For this reason, we chose the simpler approach presented in this paper.

Another approach to the problem is to start with the arrangement of the short interval moment $\int_T^{T+\Delta} |L(f,1/2+it)|^2 dt$ of \cite{ALM} and sum over a discrete set of values of $T$ (spaced by $\Delta$) to cover the dyadic interval $[T, 2T]$.  Executing this discrete sum with Poisson summation, and with some simplification, one would arrive at an expression essentially equivalent to \eqref{eq:sketchadditive} (and hence the subsequent steps would be the same).  

Finally, we mention that \cite{ALM} implicitly bounds a variant of the norm $\mathcal{N}(N', Q, k, Y)$ by an iterative Cauchy-Schwarz argument
instead of the large sieve inequality presented here. 
A back-of-the-envelope type calculation suggests that their method can be applied to deduce the same bound as Theorem \ref{thm:dualnormbound}. 
The large sieve approach here is simpler, especially because
the machinery for estimating character sums only uses Weil's bound instead of the multivariable sums in \cite{ALM} which rely on Deligne's bound.  

\subsection{Acknowledgements}
We thank Chung-Hang Kwan for helpful comments.

\section{Preliminaries and background}
In this section we collect some results from the literature that will be used later.
\subsection{Archimedean Analysis}
First we recall the definition of a family of inert functions.  Suppose $\mathcal{F}$ is an index set and $X: \mathcal{F} \rightarrow \mr_{\geq 1}$ is a function. We denote $X(i)$ as $X_i$, for $i \in \mathcal{F}$.
\begin{mydefi}
 A family of smooth functions $\{ w_i : i \in \mathcal{F} \}$ supported on a Cartesian product of dyadic intervals in $\mr_{>0}^d$ is called $X$-inert if for each $j = (j_1, \dots, j_d) \in \mz_{\geq 0}^d$ we have
 \begin{equation*}
  \sup_{i \in \mathcal{F}}
  \sup_{(x_1, \dots, x_d) \in \mr_{>0}^d}
  \frac{x_1^{j_1} \dots x_d^{j_d} |w_i^{(j_1, \dots, j_d)}(x_1, \dots, x_d)|}{ X_i^{j_1 + \dots + j_d}} < \infty.
 \end{equation*}
\end{mydefi}
For brevity, but as an abuse of terminology, we may refer to a single function being inert when we should properly say that it is part of an inert family.
In practice we will often have functions depending on a number of variables and it is unwieldy to list them all.  
Our convention will be to write $w(x_1, \dots, x_k; \cdot)$ where $\cdot$ represents some suppressed list of variables for which $w$ is inert; we will use this notation even for $k=0$.

We now state some properties of inert families. Interested readers can find the proofs of these lemmas in \cite[Sec. 2 and 3]{KPY} and \cite{BKY}.

\begin{mylemma}[Integration by parts]
\label{lemma:integrationbyparts}
Suppose that $\{ w_i(t) \}$ is a family of $X$-inert functions supported on $[Z, 2Z]$, with $w^{(j)}(t) \ll (Z/X)^{-j}$, for each $j=0,1,2, \dots$.  Suppose that $\phi$ is smooth and satisfies $\phi^{(j)}(t) \ll \frac{Y}{Z^j}$ for $j \geq 1$, for some $Y/X \geq R \geq 1$ and all $t$ in the support of $w$.  If $|\phi'(t)| \gg \frac{Y}{Z}$ for all $t \in [Z,2Z]$ then for arbitrarily large $A > 0$ we have
\begin{equation*}
 \intR w(t) e^{i \phi(t)} dt \ll_A Z R^{-A}.
\end{equation*}
\end{mylemma}

\begin{myprop}[Stationary phase]
\label{prop:statphase}
 Suppose $\{ w_i(t_1, \dots, t_d) \}$ is an $X$-inert family supported on $t_1 \asymp Z$ and $t_k \asymp X_k$ for $k=2,\dots, d$.  Suppose that on the support of $w_i$, that $\phi = \phi_i$ satisfies
 \begin{equation*}
  \frac{\partial^{a_1 + \dots + a_d}}{\partial t_1^{a_1} \dots \partial t_d^{a_d}}
  \phi(t_1, \dots, t_d) \ll \frac{Y}{Z^{a_1}} \frac{1}{X_2^{a_2} \dots X_d^{a_d}},
 \end{equation*}
for all $(a_1, \dots, a_d) \in \mathbb{Z}_{\geq 0}$.  Suppose $\phi''(t_1, t_2, \dots, t_d) \gg \frac{Y}{Z^2}$ (here and below $\phi'$ and $\phi''$ refer to the derivative with respect to $t_1$) on the support of $w$.  Moreover, suppose there exists $t_0 \in \mr$ such that $\phi'(t_0) = 0$, and that $Y/X^2 \geq R \geq 1$.  Then
\begin{equation*}
 \intR e^{i \phi(t_1, \dots, t_d)} w(t_1, \dots, t_d) dt_1
 = \frac{Z}{\sqrt{Y}} e^{i \phi(t_0, t_2, \dots, t_d)} W(t_2, \dots, t_d) + O(Z R^{-A}),
\end{equation*}
where $W = W_i$ is some new family of $X$-inert functions.
\end{myprop}

\subsection{Voronoi Summation Formula}
The $GL_3$ Voronoi summation formula was first obtained by Miller and Schmid  (\cite{MS}). We restate \cite[Theorem 1.18]{MS} as a lemma here, changing notation to match ours.

\begin{mylemma} \label{lem:gl3voronoi}
 Let $f$ be a Hecke cusp form for $SL_3(\mz)$, whose $(m,n)^{\text{th}}$ Fourier coefficient is denoted by $\lambda_f(m,n)$. Let $\psi$ be a smooth compactly-supported function on $\mathbb{R}_{>0}$. Assume $(a,c)=1$.  Then
 \begin{equation*}
\sum_n \lambda_f(1,n) e_c(an) \psi(n)
= 
\sum_{\eta \in \{\pm 1 \}}
c \sum_{n_1 | c} \sum_{n_2 > 0} \frac{\lambda_f(n_2, n_1)}{n_1 n_2} S(\eta \overline{a}, n_2 ; c/n_1) \Psi_{\eta}\Big(\frac{n_1^2 n_2}{c^3}\Big),
\end{equation*} 
where the function $\Psi_{\eta}$ is defined the same way as $F$ in \cite[Theorem 1.18]{MS}.
\end{mylemma}

Combining the work of Xiaoqing Li \cite{Li} and Blomer \cite{blomer2012subconvexity}, we have the following useful approximate integral representation for $\Psi_{\eta}$.
\begin{mylemma}\label{lem:VoronoiAsymptoticExpansion}
    Let $\varepsilon>0$ and $R\geq1$. Suppose $\psi$ has compact support on $[X,2X]$. For each $A > 0$ there exist $R$-inert functions $w = w_{A,X, \eta,\pm}$ such that 
    \begin{align*}
        \Psi_{\eta}(x) = \sum_{\pm} \frac{x^{2/3}}{X^{1/3}} \int_0^{\infty} \psi(y) e(\pm 3 x^{1/3} y^{1/3}) 
        w(x,y) dy
        %w_{A, X}^{\eta,\pm}(x y) dy 
        + O(\|\psi\|_1R^{-A}).
    \end{align*}
%    for all $x>0$.
\end{mylemma}
\begin{proof}
    For $xX\geq \sqrt{R}$, we apply the asymptotic expansion for $\Psi_{\eta}$ derived by Xiaoqing Li in \cite[(2.5), Lemma 2.1]{Li}. This yields 
    \begin{equation} 
    \label{eq:xXlarge}
        \Psi_{\eta}(x) = \sum_{\pm} \frac{x^{2/3}}{X^{1/3}} \int_0^{\infty} \psi(y) e(\pm 3 x^{1/3} y^{1/3}) 
        w(x,y) dy
        %w_{A, X, 1}^{\eta,\pm}(x y) dy 
        + O(\|\psi\|_1R^{-A}),
    \end{equation}
    where $A>0$ is arbitrarily large and 
    $w$ is $R$-inert.
    %$w_{A, X, 1}^{\eta,\pm}(\cdot)$ is some $R$-inert function. 
    
    On the other hand, for $xX\leq R$, \cite[Lemma 7]{blomer2012subconvexity} with a change of variable implies that 
    \begin{align*}
        \Psi_{\eta}(x) = \frac{x^{2/3}}{X^{1/3}} \int_0^{\infty} \psi(y) F_{X,\eta}(x, y) dy,
    \end{align*}
    for some $R$-inert function $F_{X,\eta}$. 
   We may artificially multiply by $e(\pm 3 x^{1/3} y^{1/3}) e(\mp 3 x^{1/3} y^{1/3})$, and since $xX \leq R$, the latter factor may be absorbed into the inert weight function.

The only slight logical issue remaining is that the inert function breaks into cases depending on whether $xX \geq \sqrt{R}$ or $x X \leq R$, which may appear to cause a discontinuity with respect to $x$.  However, it is easy to apply a smooth partition of unity to give a unified formula for $w$ which is $R$-inert with respect to $x$.
\end{proof}

\subsection{The Delta Method}
We state here the version of the delta symbol we use in this paper. This is a restatement of \cite[Lemma 3.1]{LeungShiftedSum}.

\begin{mylemma}\label{DeltaCor}
Let $n \in \mathbb{Z}$ be such that $|n| \ll N$, and let $C > N^{\epsilon}$. Let $U \in C_{c}^{\infty}(\mathbb{R})$ and $W \in C_{c}^{\infty}([-2, -1] \cup [1, 2])$ be nonnegative even functions such that $U(x) = 1$ for $x \in [-2, 2]$. Then  
\begin{equation*}
\delta(n = 0) = \frac{1}{C} \sum_{c = 1}^{\infty}\sum_{d = 1}^{\infty} \frac{1}{cd} 
S(0,n;c) F\left(\frac{cd}{C}, \frac{n}{cdC} \right),
\end{equation*}
where
\begin{equation}
\label{eq:Fdef}
F(x, y) \coloneqq C\left(\sum_{c = 1}^{\infty} W \left(\frac{c}{C} \right)\right)^{-1}\left(W(x) U(x) U(y)-W(y) U(x) U(y)\right)
\end{equation}
is a smooth function supported on $|x| + |y| \ll 1$ and satisfying $F^{(i,j)}(x,y) \ll_{i,j} 1$.
\end{mylemma}
We should note that Lemma \ref{DeltaCor} is the Duke–Friedlander–Iwaniec
delta method (cf. \cite{DFI}) with a simpler weight function that restricts the variable $n$ to $|n| \ll cdC$.

It will also be convenient to write $F(x,y) = F_1(x) U_1(y) + F_2(x) U_2(y)$ according to \eqref{eq:Fdef}.

\subsection{Average bounds on Fourier coefficients}
A standard result from Rankin-Selberg theory is that (cf. \cite[Theorem 7.4.9]{Goldfeld})

\begin{equation}
\label{eq:convexityGL3}
\sum_{a^2 b \leq x} |\lambda_f(a,b)|^2 \ll x.
\end{equation}
We record a variant of this which will be useful later.
\begin{mylemma}
\label{lemma:RankinSelbergBound}
We have
\begin{equation*}
\sum_{a^2 b \asymp X} \frac{|\lambda_f(a,b)|^2}{a^3 b^2} \ll X^{-1+\varepsilon}.
\end{equation*}
\end{mylemma}
\begin{proof}
Recall the Hecke relation (see \cite[Thm. 6.4.11]{Goldfeld} and use M\"obius inversion)
\begin{equation}
\label{eq:HeckeRelation}
\lambda_f(a,b) = \sum_{d|(a,b)} \mu(d) \lambda_f(a/d, 1) \lambda_f(1, b/d).
\end{equation}
Using \eqref{eq:HeckeRelation}, Cauchy's inequality, and \eqref{eq:convexityGL3}, we deduce
\begin{align*}
\sum_{a^2 b \asymp X} \frac{|\lambda_f(a,b)|^2}{a^3 b^2} &\ll
X^{\varepsilon}
\sum_{a^2 b \asymp X} \frac{1}{a^3 b^2} \sum_{d | (a,b)} |\lambda_f(a/d, 1)|^2 \cdot |\lambda_f(1, b/d)|^2
\\
&\ll 
X^{\varepsilon} \sum_{d^3 \ll X} \frac{1}{d^5} \sum_{a^2 \ll X/d^3} \frac{|\lambda_f(a,1)|^2}{a^3} \sum_{b \asymp \frac{X}{a^2 d^3}} \frac{|\lambda_f(1,b)|^2}{b^2}
\\
&\ll X^{-1+\varepsilon} \sum_{d^3 \ll X} \frac{1}{d^2} \sum_{a^2 \ll X/d^3} \frac{|\lambda_f(a,1)|^2}{a}
\ll X^{-1 +\varepsilon}. \qedhere
\end{align*}
\end{proof}

\subsection{Large Sieve Inequality}
We recall a hybrid form of the large sieve inequality:
\begin{mylemma} \cite[Theorem 2]{Gallagher}
\label{lem:largesieve}
Let $Y \geq 1$, and let $d$ be a positive integer.  Then for arbitrary complex numbers $a_n$ we have
\begin{equation*}
\sum_{\psi \shortmod{d}}
\thinspace
\int_{-Y}^{Y}
\Big|\sum_{n \leq N} a_n \psi(n) n^{iy} \Big|^2 dy
\ll (  d  Y + N) \sum_{n \leq N} |a_n|^2.
\end{equation*}
\end{mylemma}

\subsection{Twisted Kloosterman sum bounds}
For a Dirichlet character $\chi \pmod{c}$, let
\begin{equation*}
    S_{\chi}(m,n;c) = \sumstar_{t \shortmod{c}} \overline{\chi}(t) e_c(mt + n \overline{t}).
\end{equation*}
It is unfortunate that the twisted Kloosterman sum does not always exhibit square-root cancellation.  The following collection of bounds is satisfactory for our later purposes.
\begin{mylemma}
\label{lemma:WeilTwisted}
Suppose $p$ is prime and $(mn, p) = 1$.
\begin{enumerate}
    \item We have $|S_{\chi}(m,n;p)| \leq 2 p^{1/2}$.
    \item Suppose $p$ is odd,  $j \geq 2$, and $\cond(\chi) \leq p^{j-1}$.  Then
     $|S_{\chi}(m,n;p^j)| \leq 2 p^{j/2}$.
\item Suppose $j \geq 2$.  Then $|S_{\chi}(m,n;p^j)| \leq 2 p^{j-\frac12}$.
\item Suppose $p=2$, $j \geq 2$, and  $\cond(\chi) \leq 2^{j-1}$.  Then 
$|S_{\chi}(m,n;2^j)| \leq 4 \cdot 2^{j/2}$.
\end{enumerate}
\end{mylemma}
Case (1) was proved by Weil.  Almost all the remaining cases follow from \cite[Prop. 9.7, 9.8]{KnightlyLi}.  The only cases left uncovered are case (3) with $p=2$ and case (4) with $\cond(\chi) = 2^{j-1}$ (Knightly and Li handle $\cond(\chi) \leq 2^{j-2}$).
In this latter case, we use \cite[(7.1.20)]{Kroesche}, which gives the bound
$|S_{\chi}(m,n;2^j)| \leq 2 \cdot 2^{j/2 + \delta}$, where 
$\delta = \frac12 \nu_2(2^{2j} + 4mn)$.  Here $\nu_2(\cdot)$ denotes the usual $2$-adic valuation, and since $(mn,2) = 1$ by assumption we have $\delta=1$.
Finally, case (3) with $p=2$ is worse than the trivial bound of $2^j$.

\subsection{An elementary estimate}
We state the following for later use.
\begin{mylemma}
\label{lemma:radicalsum}
    Let $\mathrm{rad}(n) = \prod_{p|n} p$.  Then for any $q \geq 1$ we have
    \begin{equation*}
        \sum_{n \leq X} \mathrm{rad}(nq)^{-1/2} \ll \mathrm{rad}(q)^{-1/2} X^{1/2} (qX)^{\varepsilon}.
    \end{equation*}
\end{mylemma}
\begin{proof}
Using Rankin's trick, the sum on the left hand side above is
\begin{equation*}
    \leq  \sum_{n=1}^{\infty} \frac{X^{1/2+\varepsilon}}{n^{1/2+\varepsilon}  \mathrm{rad}(nq)^{1/2}}
    = X^{1/2+\varepsilon} \prod_{p \nmid q} (1+p^{-1-\varepsilon} + \dots)
    \prod_{p|q} p^{-1/2} (1+p^{-1/2} + \dots).
\end{equation*}
The resulting product is easily seen to satisfy the claimed bound. 
\end{proof}

\section{Initial setup}
\subsection{Approximate functional equation and dissection}
By an approximate functional equation, dyadic partition of unity, and use of Cauchy's inequality, we have
\begin{equation*}
 \int_{T}^{2T} |L(f, 1/2+it)|^2 dt
 \ll 
 T^{\varepsilon} \max_{1 \ll N \ll T^{3/2+\varepsilon}}
 \int_{T}^{2T} \Big| \sum_{n \asymp N} \frac{\lambda_f(n)}{n^{1/2+it}} V(n) \Big|^2 dt. 
\end{equation*}
Here 
$\lambda_f(n):=\lambda_f(1,n)$ denotes the $(1,n)^{\text{th}}$ Fourier coefficient of $f$, and $V= V_{f,N}$ is a smooth function satisfying $V^{(j)}(x) \ll_j N^{-j}$ for $x \asymp N$.

Note that by the mean value theorem for Dirichlet polynomials, we have
\begin{equation}
\label{eq:MVTDPbound}
 \int_{T}^{2T} \Big| \sum_{n \asymp N} \frac{\lambda_f(n)}{n^{1/2+it}} V(n) \Big|^2 dt
 \ll (T+N) N^{\varepsilon},
\end{equation}
which is satisfactory for Theorem \ref{thm:mainthm} for $N \ll T^{4/3}$.  Hence we may assume
\begin{equation}
\label{eq:Nsize}
 T^{4/3} \ll N \ll T^{3/2+\varepsilon}.
\end{equation}
Let 
\begin{equation*}
 I(T) = I(T,N) = N^{-1} 
 \intR \omega \Big(\frac{t}{T} \Big) \Big|\sum_n \lambda_f(n) n^{-it} w_N(n) \Big|^2 dt,
\end{equation*}
where $\omega$ is a fixed nonnegative smooth bump function on $\mathbb{R}^+$ and $w_N$ is a smooth bump function supported on $[N, 2N]$, satisfying $w_N^{(j)}(x) \ll N^{-j}$.
The previous discussion shows
\begin{equation}
\label{eq:ITjsum}
 \int_T^{2T} |L(f,1/2+it)|^2 dt 
 \ll T^{4/3+\varepsilon}
 + \max_{T^{4/3} \ll N \ll T^{3/2+\varepsilon}} I(T,N),
\end{equation}
for appropriate choices of $\omega$ and $w_N$. To prove Theorem \ref{thm:mainthm}, it suffices to show
\begin{equation*}
    I(T) \ll T^{4/3 + \varepsilon},
\end{equation*}
for all $N$ satisfying \eqref{eq:Nsize}.

\subsection{Applying the delta symbol}\label{subsect:ApplyingDelta}

Squaring out and integrating, we have
\begin{equation*}
 I(T) =  \frac{T}{N} 
 \sum_{m,n} \lambda_f(m) \overline{\lambda_f}(n) \widehat{\omega}\Big(\frac{T}{2\pi} \log\frac{m}{n}\Big) w_N(m) w_N(n).
\end{equation*}
Next we write $m=n+k$, getting
\begin{equation*}
I(T) = \frac{T}{N} 
 \sum_{m,n,k} \lambda_f(m) \overline{\lambda_f}(n) \widehat{\omega}\Big(\frac{T}{2\pi} \log\frac{n+k}{n}\Big) w_N(m) w_N(n) \delta(n+k-m).
\end{equation*}
By a Taylor expansion, we have
\begin{equation*}
 \widehat{\omega}\Big(\frac{T}{2\pi} \log\frac{n+k}{n}\Big) %= \widehat{\omega}\Big(\Delta \frac{k}{n} (1 + O(k/n))\Big)  
 = \widehat{\omega}\Big(\frac{T k}{2\pi n}\Big) + 
 \frac{1}{T} 
 \Big(\frac{-T^2 k^2}{4\pi n^2}\Big) \widehat{\omega}'\Big(\frac{T k}{2\pi n}\Big) + \dots.
\end{equation*}
%and we may correspondingly write $I(T, \Delta) = I_0(T, \Delta) + \Delta^{-1} I_1(T, \Delta) + \dots$.  All the lower-order terms may be estimated the same way as $I_0$, so we will focus on $I_0$ going forward.
This shows
$I(T) = I_0(T) + O(N T^{-1+\varepsilon})$, where
\begin{equation*}
I_{0}(T) :=
\frac{T}{N} 
 \sum_{m,n,k} \lambda_f(m) \overline{\lambda_f}(n) \widehat{\omega}\Big(\frac{T k}{2\pi n}\Big) w_N(m) w_N(n) \delta(n+k-m).
\end{equation*}
Note that $ N/T \ll T^{1/2+\varepsilon}$, so this error term is more than sufficient for the purposes of Theorem \ref{thm:mainthm}.
Next we apply Lemma \ref{DeltaCor} for $\delta(m-n-k=0)$.
We will choose $C$ optimally at the end as $C = N^{1/2} T^{-1/6}$, but for now we only assume
\begin{equation}
\label{eq:RsizeDef}
N T^{-1+\varepsilon} \leq C \leq N^{1/2}T^{-\varepsilon}.
\end{equation}

In summary, we have shown the following.
\begin{myprop}
\label{prop:I00formula}
 We have
 \begin{equation}
 \label{eq:I00def}
  I_{0}(T) = 
  \sum_{i=1}^{2}
  \frac{T}{NC}
 \sum_{cd \leq C} \frac{1}{cd} F_i\left(\frac{cd}{C} \right)
 \thinspace \thinspace
 \sumstar_{h \shortmod{c}} \mathcal{S},
 \end{equation}
with
\begin{multline}
\label{eq:Sdef}
 \mathcal{S} = \mathcal{S}_i = 
 \sum_{m,n,k} \lambda_f(m) e_c(-hm) \overline{\lambda_f}(n) e_c(hn) \widehat{\omega}\Big(\frac{T k}{2 \pi n}\Big) 
 \\
 \times e_c(hk)
 w_N(m) w_N(n)
  U_i \left(
 \frac{n+k-m}{cdC}\right).
 \end{multline}
\end{myprop}
In what follows we replace $U_i$ by $U$, since all the bounds will be independent of $i$.

\section{Summation formulas}\label{sect:DualSummations}
\subsection{Result statement}
To $\mathcal{S}$
we will apply Poisson summation in $k$ and $GL_3$ Voronoi summation in each of $m$, $n$.
The end result of this step is summarized in the following result, whose proof takes up the rest of this section.
To state this result, we introduce some notation.  
We let $T' = \frac{T}{2 \pi}$, and
for integer $k' \geq 1$ we let
\begin{equation}
\label{eq:nudef}
 \nu = \frac{c T'}{k'}.
\end{equation}

\begin{myprop}
\label{prop:SafterVoronoi}
We have
\begin{equation}
\label{eq:I00TDeltaBound}
I_{0}(T)
\ll T^{\varepsilon}
\max_{\substack{Q \ll C \\ \eta \in \{ \pm 1 \} }}
\sum_{d \ll C/Q}
\max_{\substack{N' \ll \frac{N^2}{d^3 C^3} T^{\varepsilon}}}
N\left(\frac{N'^2}{NQ^3}\right)^{2/3}  
\int_{t \asymp 1}   \sum_{c \asymp Q} 
 \sum_{k' \asymp \frac{QT}{N} } 
|\mathcal{A} |^2
 + O(T^{-100}),
\end{equation}
where 
\begin{equation}
\label{eq:A1def}
\mathcal{A} = 
\mathop{\sum_{n_1 | c} \sum_{n_2}}_{n_1^2 n_2 \asymp N'} 
\frac{\lambda_f(n_1, n_2)}{n_1 n_2}
S(\eta   \overline{k'}, n_2 ;c/n_1) 
e_c\Big(3   (n_1^2 n_2 \nu t)^{1/3}  \Big)
w(\cdot)
,
\end{equation}
for $w$ some $T^\varepsilon$-inert function.
\end{myprop}

\subsection{Poisson}
For fixing the sizes of some parameters, it is helpful to apply Poisson summation first in isolation.
\begin{mylemma}
%Recall $T' = \frac{T}{2 \pi}$.
Poisson summation applied to the inner sum over $k$ gives
 \begin{equation}
 \label{eq:PoissonSolo}
\mathcal{S} = \sum_{k' \equiv h \shortmod{c}} 
 \sum_{m,n} \lambda_f(m) e_c(-hm) \overline{\lambda_f}(n) e_c(hn) 
  w_N(m) w_N(n)
   \widehat{H}\Big(-\frac{k'}{c}\Big),
 \end{equation}
 where
 $ H(z) = \widehat{\omega}(T' z/n) U(\frac{n-m+z}{cdC})$.
In \eqref{eq:PoissonSolo}, the sum may be restricted to
$k'>0$ and 
\begin{equation}
\label{eq:k'crestriction}
    \frac{k'}{c} \asymp \frac{T}{N},
\end{equation}
with a small error term of size $\ll T^{-2024}$.
\end{mylemma}
 Note that \eqref{eq:k'crestriction} implies that
 $c \asymp \frac{k' N}{T}$, so using \eqref{eq:Nsize} we may assume
\begin{equation}
\label{eq:clowerbound}
c \gg \frac{N}{T} \gg T^{1/3}.
\end{equation}

\begin{proof}
Isolating the sum over $k$ in \eqref{eq:Sdef}, and applying Poisson summation, we have
\begin{equation*}
 \sum_k \widehat{\omega}\Big(\frac{T' k}{n}\Big) e_c(hk) U\left(\frac{n+k-m}{cdC}\right)
 = \sum_{k' \equiv h \shortmod{c}} \widehat{H}\Big(-\frac{k'}{c}\Big).
\end{equation*}
Now note that by Taylor expansion, 
\begin{align}
\label{Taylor}
    \widehat{H}(-y)=&\intR \widehat{\omega}(T'z/n)U\left(\frac{n-m+z}{cdC}\right)e(yz)dz\nonumber\\
    =&\frac{n}{T'}\sum_{j=0}^J \frac{n^j}{j!(cdCT')^j}U^{(j)}\left(\frac{n-m}{cdC}\right)\intR \widehat{\omega}(x)x^je\left(\frac{nxy}{T'}\right)dx+O\left(\frac{N}{T}\left(\frac{N}{CT}\right)^{J+1}\right)\\
    =&\frac{n}{T'}\sum_{j=0}^J \frac{n^j}{j!(icdCT)^j}U^{(j)}\left(\frac{n-m}{cdC}\right)\omega^{(j)}\left(\frac{ny}{T'}\right)+O\left(\frac{N}{T}\left(\frac{N}{CT}\right)^{J+1}\right)\nonumber,
\end{align}
for any $J\geq0$. This gives an asymptotic expansion by the assumption \eqref{eq:RsizeDef}, so the sum may be truncated with a small error term. Hence \eqref{eq:k'crestriction} follows from the support of $\omega$.
\end{proof}

\subsection{Voronoi}
Next we examine the effect of the $GL_3$ Voronoi formula applied to both the $m$ and $n$ sums in $\mathcal{S}$. Using Lemmas \ref{lem:gl3voronoi}
 and \ref{lem:VoronoiAsymptoticExpansion}, we obtain
\begin{multline}
\label{eq:SandIfirstappearance}
 \mathcal{S} = c^2
  \sum_{\eta_1,\eta_2, \eta_3, \eta_4 \in \{ \pm 1 \}}\sum_{k' \equiv h \shortmod{c}} 
 \sum_{\substack{n_1 | c \\ m_1 | c}}
 \sum_{n_2, m_2} 
 \frac{\lambda_f(m_2, m_1) \lambda_f(n_1, n_2)}{m_1 m_2 n_1 n_2}
\frac{\Big(\frac{n_1^2 n_2}{c^3} \Big)^{2/3} \Big(\frac{m_1^2 m_2}{c^3} \Big)^{2/3}}{N^{2/3}}
 \\
\times S(\eta_3 \overline{h}, m_2;c/m_1) S(\eta_4 \overline{h}, n_2 ;c/n_1) 
\cdot
\mathcal{I} + O\left(T^{-2024}\right),
 \end{multline}
for any $A>0$,
where
\begin{multline} 
\label{eq:mathcalI}
 \mathcal{I} 
 = 
 \intR \intR \intR 
w_N(x) w_N(y)\widehat{\omega}\Big(\frac{T'z}{x}\Big)
w(x,y, \cdot)
\\
\times U\left(\frac{x-y+z}{cdC}\right) e\left(\eta_1 \frac{3 (y m_1^2 m_2)^{1/3}}{c}+\eta_2\frac{3 (x n_1^2 n_2)^{1/3}}{c}+\frac{k' z}{c}\right) 
  dx dy dz.
\end{multline}
Note that in \eqref{eq:mathcalI} we can restrict the integral to $\frac{z}{x} \ll T^{-1+\varepsilon}$.
It is convenient here to apply a dyadic partition of unity to $c$, $m_1^2 m_2$, and $n_1^2 n_2$.  Say that a piece of the partition implies the conditions
\begin{equation*}
 c \asymp Q,
 \qquad
 m_1^2 m_2 \asymp M', 
 \qquad
 n_1^2 n_2 \asymp N'.
\end{equation*}
\iffalse
In this notation, the previous condition $xX \gg T^{\varepsilon}$ is equivalent to
\begin{equation}
\label{eq:M'N'lowerbound}
\frac{M' N}{Q^3} \gg T^{\varepsilon},
\qquad
\text{and}
\qquad
\frac{N' N}{Q^3} \gg T^{\varepsilon}.
\end{equation}
\fi
Then we can simplify our expression for $\mathcal{S}$ to take the form
\begin{multline}
\label{eq:SandIsecondappearance}
 \mathcal{S} = 
 \sum_{\eta_1,\eta_2, \eta_3, \eta_4 \in \{ \pm 1 \}}
 \sum_{\substack{N', M', Q \\ \text{dyadic}}}
 \Big(\frac{N' M'}{N Q^3} \Big)^{2/3} 
  \sum_{k' \equiv h \shortmod{c}} 
 \sum_{\substack{n_1 | c \\ m_1 | c}}
 \sum_{\substack{n_1^2 n_2 \asymp N' \\ m_1^2 m_2 \asymp M'}} 
 \frac{\lambda_f(m_2, m_1) \lambda_f(n_1, n_2)}{m_1 m_2 n_1 n_2}
 \\
\times S(\eta_3 \overline{h}, m_2;c/m_1) S(\eta_4 \overline{h}, n_2 ;c/n_1) 
\cdot
\mathcal{I} + O\left(T^{-2024}\right),
 \end{multline}
and where now $\mathcal{I}$ is slightly modified by absorbing the dyadic partitions into the inert weight function.

We next focus on the asymptotic behavior of $\mathcal{I}$. 
\begin{mylemma}\label{lemma:IntegralAnalysis}
Let $D = \eta_1 (m_1^2 m_2)^{1/3} + \eta_2 (n_1^2 n_2)^{1/3}$. There
exists a $T^{\varepsilon}$-inert function $w(\cdot)$ of all relevant parameters such that
\begin{equation}
 \label{eq:IformulaDecent}
\mathcal{I} 
 = 
 cdC
\frac{N^2}{T} 
% \widehat{F}\Big(\frac{cd}{C}, -(m_1^2 m_2)^{1/3} \nu^{-2/3} dC\Big)
\int_{t \asymp 1}
e_c\Big(3 D (\nu t)^{1/3}\Big)
 w(\cdot)
 dt
 +O(T^{-2024}).
\end{equation}
\end{mylemma}
%\my{Revisit this proof in light of the uniform version of Xiaoqing's integral representation...}
 \begin{proof}
As a first step,  we will constrain the sizes of $M'$ and $N'$.  We claim that $\mathcal{I}$ is very small unless
\begin{equation}
\label{eq:n1n2m1m2bound}
 M' + N' 
 %\ll T^{\varepsilon} \Big(\frac{N^2}{d^3C^3} + \frac{Q^3}{N}\Big)
 \ll T^\varepsilon\frac{N^2}{d^3C^3}. %\asymp T^{\varepsilon} \frac{N^{1/2} \Delta^{3/2}}{d^3}.
\end{equation}
To see this, consider first the $y$-integral.  
We use Lemma \ref{lemma:integrationbyparts}
(integration by parts)
 with $Y = (\frac{N M'}{Q^3})^{1/3}$, $X = T^{\varepsilon}(1 + \frac{N}{QdC})$, $Z=N$.
 Note that using \eqref{eq:RsizeDef} we have $\frac{N}{QdC} \gg \frac{N}{C^2} \gg T^{\varepsilon}$.
Hence this integral is very small unless
\begin{equation}
\label{eq:M'boundimplication}
 \left(\frac{NM'}{Q^3}\right )^{\frac13} \ll  \frac{N}{QdC} T^{\varepsilon} 
 \Longleftrightarrow 
 M'  \ll T^{\varepsilon}  \frac{N^2}{d^3C^3}. %\asymp T^{\varepsilon}  \frac{N^{1/2} \Delta^{3/2}}{d^3}.
\end{equation}

The case with the $x$ integral is essentially the same.  The only difference is the presence of $\widehat{\omega}(T'z/x)$, but this does not change the value of $X$, so the integration by parts argument leads to the same truncation point on $N'$ as for $M'$.

Note that since we may assume $z \ll N T^{-1+\varepsilon}$ then $z = o(cdC)$ by \eqref{eq:RsizeDef}, so 
in practice we have $U(\frac{x-y+z}{cdC}) = U(\frac{x-y}{cdC}) + \dots$ as in \eqref{Taylor}, where the lower-order terms may be absorbed into the inert weight function.  This means the $z$-integral may be easily evaluated to give
\begin{equation*}
 \intR \widehat{\omega}\Big(\frac{T'z}{x}\Big) 
 U\Big(\frac{x-y+z}{cdC}\Big) 
 e\Big(\frac{k' z}{c}\Big) dz
 = \frac{x}{T'} U\Big(\frac{x-y}{cdC}\Big) 
 \omega \Big(\frac{x k'}{c T'} \Big) + \dots.
\end{equation*}
We can write $\frac{x}{T'} = \frac{N}{T} \cdot \frac{2\pi x}{N}$ and absorb the factor $\frac{2\pi x}{N}$ into the inert weight function.
We also change variables $y=x+u$.
Summarizing the progress so far, we have
\begin{multline} 
\label{eq:mathcalImodified}
 \mathcal{I} 
 = 
 \frac{N}{T} \intR \intR 
w_N(x) w_N(x+u)
   \omega \Big(\frac{x k'}{c T'} \Big)
w(\cdot)
\\
\times U\left(\frac{-u}{cdC}\right) e\left(\eta_1 \frac{3 ((x+u) m_1^2 m_2)^{1/3})}{c}+\eta_2\frac{3 (x n_1^2 n_2)^{1/3}}{c} \right) 
  du 
  dx + \dots.
\end{multline}

Now
the basic idea is that the $u$-integral is contained in a short enough interval that the exponential terms can be linearly approximated, though this takes some work to develop in detail.  
Note that
\begin{equation*}
 3 \frac{((x+u) m_1^2 m_2)^{1/3}}{c}
 = 3 \frac{(x m_1^2 m_2)^{1/3}}{c}
 \left(1 + \frac{u}{x}\right)^{1/3}.
\end{equation*}
We then write $(1+ \frac{u}{x})^{1/3} = 1 + \frac13 \frac{u}{x} + O(N^{-2} u^2)$.  
Since $u \ll c d C$,
and so using \eqref{eq:n1n2m1m2bound} and \eqref{eq:clowerbound}, this error term is of size
\begin{equation*}
\frac{(x m_1^2 m_2)^{1/3}}{c} \cdot \frac{u^2}{x^2} \ll 
 \frac{N T^{\varepsilon}}{cdC} \frac{(cdC)^2}{N^2} 
 \ll \frac{T^{\varepsilon} C^2}{N}\ll 1.
\end{equation*}
This may be absorbed into the inert weight function.  Recall $D = \eta_1 (m_1^2 m_2)^{1/3} + \eta_2 (n_1^2 n_2)^{1/3}$.  
Substituting $y=x+u$ gives that 
\begin{align}
\label{eq:exponentialtermwithcuberoots}
e_c(3(\eta_1 (y m_1^2 m_2)^{1/3} + \eta_2 ( x n_1^2 n_2)^{1/3})) 
= e_c\left(3 D x^{1/3} +\eta_1 u \frac{(m_1^2 m_2)^{1/3}}{x^{2/3}}\right) w(\cdot),
\end{align}
where $w(\cdot)$ is $1$-inert.
Applying these approximations into $\mathcal{I}$, we obtain
\begin{multline*}
 \mathcal{I} 
 = \frac{N}{T}
 \intR \intR 
e_c(3 D x^{1/3})
w_N(x) w_N(x+u) 
   \omega \Big(\frac{x k'}{c T'} \Big)
   \\
   \times
e_c\Big(\eta_1 u \frac{(m_1^2 m_2)^{1/3}}{x^{2/3}}\Big)
U\Big(\frac{-u}{cdC}\Big) 
w(\cdot)
 du dx,
\end{multline*}
up to a small error term.  The $u$-integral returns the Fourier transform of an inert function, has size $cdC$, and essentially constrains $M'$ according to \eqref{eq:n1n2m1m2bound}.  Finally, for the $x$-integral we change variables
$x = \nu t$, where $\nu = \frac{cT'}{k'}$, giving
\begin{equation*}
 \mathcal{I} 
 = 
 cdC
\frac{N^2}{T} 
\intR  e_c(3D (\nu t)^{1/3}) w(\cdot) 
\omega(t)
dt
  + O(T^{-2024}). \qedhere
\end{equation*}
\end{proof}

Now we are ready to prove Proposition \ref{prop:SafterVoronoi}.  We pick up with \eqref{eq:SandIsecondappearance}, insert \eqref{eq:IformulaDecent}, and recall the definition of $D = \eta_1 (m_1^2 m_2)^{1/3} + \eta_2 (n_1^2 n_2)^{1/3}$ to separate these variables.  Hence
\begin{multline*}
 \mathcal{S} = 
   cdC
\frac{N^2}{T}  
 \sum_{\eta_1,\eta_2, \eta_3, \eta_4 \in \{ \pm 1 \}}
 \sum_{\substack{N', M', Q \\ \text{dyadic}}}
\Big(\frac{N' M'}{N Q^3} \Big)^{2/3} 
    \int_{t \asymp 1}
  \sum_{k' \equiv h \shortmod{c}} 
  \\
\Big[ \sum_{\substack{n_1^2 n_2 \asymp N' \\ n_1 | c}} 
 \frac{\lambda_f(n_1, n_2)}{ n_1 n_2} S(\eta_4 \overline{h}, n_2 ;c/n_1) 
 e_c\Big (3 \eta_2 (n_1^2 n_2 \nu t)^{1/3} \Big)  w(n_1^2 n_2, \cdot) \Big]
 \\
\times \Big[
 \sum_{\substack{m_1^2 m_2 \asymp M' \\ m_1|c}} 
 \frac{\lambda_f(m_2, m_1)}{m_1 m_2}
S(\eta_3 \overline{h}, m_2;c/m_1) 
e_c\Big (3 \eta_1  (m_1^2 m_2 \nu t)^{1/3}  \Big)
%\times \widehat{F}\Big(\frac{cd}{C}, -(m_1^2 m_2)^{1/3} \nu^{-2/3} dC\Big)
%\omega(-t)
% \omega\Big(\eta_1 \frac{ (\nu m_1^2 m_2)^{1/3} }{c\Delta } - t \Big)
w(m_1^2 m_2, \cdot)
 \Big]
 dt.
 \end{multline*}
Now substituting this into \eqref{eq:I00def} gives
 \begin{multline*}
  I_{0}(T) = 
 \frac{\Delta N}{T} 
 \sum_{\substack{Q \ll C \\ Q \text{ dyadic}}}
 \sum_{d \ll C/Q}
\sum_{\substack{M', N' \ll \frac{N^2}{d^3 C^3} T^{\varepsilon} \\M', N' \text{dyadic}}}
  \Big(\frac{N' M'}{N Q^3} \Big)^{2/3} 
  \sum_{\eta_1,\eta_2, \eta_3, \eta_4 \in \{ \pm 1 \}}
  \int_{t \asymp 1}
\sum_{k' \asymp \frac{QT}{N} } 
\\
      \sum_{\substack{c \asymp Q \\ (c, k') = 1}} 
    \Big[ \sum_{\substack{n_1^2 n_2 \asymp N' \\ n_1 | c}} 
 \frac{\lambda_f(n_1, n_2)}{ n_1 n_2} S(\eta_4 \overline{k'}, n_2 ;c/n_1) 
 e_c\Big (3 \eta_2 (n_1^2 n_2 \nu t)^{1/3} \Big) 
 w(n_1^2 n_2, \cdot)
 \Big]
 \\
 \times \Big[
 \sum_{\substack{m_1^2 m_2 \asymp M' \\ m_1|c}} 
 \frac{\lambda_f(m_2, m_1)}{m_1 m_2}
S(\eta_3 \overline{k'}, m_2;c/m_1) 
e_c\Big (3 \eta_1(m_1^2 m_2 \nu t)^{1/3}  \Big)
%\\
%\times \widehat{F}\Big(\frac{cd}{C}, -(m_1^2 m_2)^{1/3} \nu^{-2/3} dC\Big)
%\omega(-t)
% \omega\Big(\eta_1 \frac{ (\nu m_1^2 m_2)^{1/3} }{c\Delta } - t \Big)
w(m_1^2 m_2, \cdot)
 \Big]
 dt.
 \end{multline*}
Finally, we apply the
AM-GM inequality in the form $|\sum_{n_1, n_2} |\times |\sum_{m_1, m_2} | \ll |\sum_{n_1,n_2}|^2 + |\sum_{m_1, m_2}|^2$.  By symmetry, each of these two terms leads to the same final bound.
We may also reduce to the case $\eta_2 = 1$ by replacing $f$ by $\overline{f}$ if necessary.
This proves Proposition \ref{prop:SafterVoronoi}.

\subsection{Simplifications}
We continue from the statement of Proposition \ref{prop:SafterVoronoi}.
Our goal is to present a processed form of the bound which is simplified in some aspects, and is in a more ready-made form for the upcoming steps.
Define the expression $\mathcal{J}^{+} = \mathcal{J}^{+}(N')$ by
 \begin{equation}
 \label{eq:JplusDef}
  \mathcal{J}^{+}
  = 
  \sum_{n_1} \frac{1}{n_1^2}
 \sum_{k' \asymp \frac{QT}{N}} \sum_{\substack{b \asymp \frac{Q}{n_1} \\ (b, k') = 1}} \int_{v \asymp 1}
 \Big| 
 \sum_{n_2 \asymp \frac{N'}{n_1^2}} 
\frac{\lambda_f(n_1, n_2)}{n_2}
S(\eta  \overline{k'}, n_2 ;b) 
e\Big(\frac{3  (Nn_1^2 n_2)^{1/3}}{Q}  v \Big)
w(\cdot)
\Big|^2 dv.
 \end{equation}
We may as well define at this time the multiplicative version
$\mathcal{J}^{\times}$ by
\begin{equation}
\label{eq:JtimesDef}
  \mathcal{J}^{\times}(N')
  = 
  \sum_{n_1} \frac{1}{n_1^2}
 \sum_{k' \asymp \frac{QT}{N}} \sum_{\substack{b \asymp \frac{Q}{n_1} \\ (b, k') = 1}}
 \int_{y \asymp \Phi}
 \Big| 
 \sum_{n_2 \asymp \frac{N'}{n_1^2}} 
\frac{\lambda_f(n_1, n_2)}{n_2}
S(\eta  \overline{k'}, n_2 ;b) 
n_2^{iy} 
w(\cdot)
\Big|^2 dy.
 \end{equation}
 
\begin{myprop}
\label{prop:JwrtJplus}
We have
\begin{equation}
\label{eq:JboundNonOsc}
 I_{0}(T)
\ll 
T^{\varepsilon}\max_{\eta \in \{\pm 1\}}
\max_{\substack{Q \ll C }}
\sum_{d \ll C/Q}\max_{\substack{N' \ll \frac{N^2}{d^3 C^3} T^{\varepsilon} }}
\frac{ N^{1/3}}{Q^2}(N')^{4/3} 
\cdot
\mathcal{J}^{+}(N').
\end{equation}
% Moreover, if $\Phi' \ll T^{\varepsilon}$, then $\mathcal{J}$ is very small unless $\Phi \ll T^{\varepsilon}$.  In addition, if $\Phi' \gg T^{\varepsilon}$ then $\mathcal{J}$ is very small unless $N' \asymp M'$.
\end{myprop} 
\begin{proof}
We first use Cauchy's inequality to take the sum over $n_1$ (in $\mathcal{A}$, in \eqref{eq:I00TDeltaBound}) to the outside.  We then interchange the orders of summation and change variables $c = bn_1$.  
Finally, we change variables $v=\frac{Q}{c}\left(\frac{\nu t}{N}\right)^{1/3}$, which maintains an interval of size $\asymp 1$ (and we can over-extend to an interval independent of all parameters), to 
give \eqref{eq:JboundNonOsc}.
\end{proof}

Our next step is to relate $\mathcal{J}^{+}$ and $\mathcal{J}^{\times}$, since we prefer to work with the multiplicative characters. Let
$\Phi$ denote the size of the phase in \eqref{eq:JplusDef}, namely, 
\begin{equation}
\label{eq:PhiDef}
\Phi = \Phi(N') = \frac{(N' N)^{1/3}}{Q}.
\end{equation}
\begin{myprop}
\label{prop:JplusJtimesetc}
Let $\mathcal{J}^{+}$ be defined by \eqref{eq:JplusDef}.  Suppose $\Phi=\Phi(N') \gg T^{\varepsilon}$.
Then $\mathcal{J}^{+}(N') \ll \Phi^{-1}T^{\varepsilon} \cdot \mathcal{J}^{\times}(N') + O(T^{-100})$.
\end{myprop}
\begin{proof}
 For $X, n_2 > 0$ let
 \begin{equation*}
  h(v, n_2, \cdot) = e(X n_2^{1/3} v) w(v,\cdot),
 \end{equation*}
where $ X n_2^{1/3} \asymp \Phi  \gg T^{\varepsilon}$.  By Mellin inversion,
\begin{equation*}
 h(v,n_2, \cdot) = \int_{-\infty}^{\infty} \widetilde{h}(-iy, \cdot) (X n_2^{1/3} v)^{iy} w(v, \cdot) dy,
\end{equation*}
where
\begin{equation*}
 \widetilde{h}(-iy, \cdot) = \frac{1}{2 \pi} \int_0^{\infty} h(t, n_2, \cdot) (X n_2^{1/3} t)^{-iy} w(t, \cdot) \frac{dt}{t}.
\end{equation*}
The standard integration by parts argument shows that $\widetilde{h}(-iy, \cdot)$ is small outside of $y \asymp \Phi$.
Moreover, the stationary phase bound implies $|\widetilde{h}(-iy, \cdot)| \ll \Phi^{-1/2}$.  Hence for arbitrary coefficients $a_n$ we have
\begin{multline*}
 \int_{v \asymp 1} \Big|\sum_{n_2} a_{n_2} h(v, n_2, \cdot) \Big|^2 dv
 \\
 = \int_{v \asymp 1} \Big|\int_{y \asymp \Phi} \widetilde{h}(-iy, \cdot) (Xv)^{iy} \sum_{n_2} a_{n_2} n_2^{iy/3} dy \Big|^2 dv + O( T^{-200}\| a_{n_2} \|^2 ).
\end{multline*}
To analyze this further, first attach a smooth nonnegative weight function $w(\cdot)$ to the outer $v$-integral, and then square it out.  We encounter an integral of the form $\int_{v} w(v) v^{iy_1 - iy_2} dv$, which localizes on $|y_1 - y_2| \ll T^{\varepsilon}$.  Hence the above expression is bounded by
\begin{equation*}
 \int_{\substack{|y_1 - y_2| \ll T^{\varepsilon} \\ y_1, y_2 \asymp \Phi }} |\widetilde{h}(-iy_1 , \cdot) \widetilde{h}(-iy_2 , \cdot)|
 \Big| \sum_{n_2, n_2'} a_{n_2} \overline{a_{n_2'}} n_2^{iy_1/3} (n_2')^{-iy_2/3} \Big| dy_1 dy_2
+ O(T^{-200} \| a_{n_2} \|^2) 
 .
\end{equation*}
Next we use the AM-GM inequality and symmetry to bound this in turn by
\begin{equation*}
 \int_{\substack{|y_1 - y_2| \ll T^{\varepsilon} \\ y_1, y_2 \asymp \Phi }} (|\widetilde{h}(-iy_1 , \cdot)|^2 + |\widetilde{h}(-iy_2 , \cdot)|^2) 
 \Big| \sum_{n_2} a_{n_2} n_2^{iy_1/3}\Big|^2 dy_1 dy_2 + O(T^{-200} \| a_{n_2} \|^2).
\end{equation*}
We use that $|\widetilde{h}|^2 \ll \Phi^{-1}$, and that the integral over the ``other'' variable is at most $T^{\varepsilon}$.  We finally change variables $y_1 \rightarrow 3 y_1$, giving the claimed bound once we make a suitable choice of the coefficients $a_{n_2}$.
\end{proof}

We have a similar result in the non-oscillatory range with $\Phi \ll T^{\varepsilon}$. 
\begin{myprop}
\label{prop:JplusJtimesetcNonOsc}
Let $\mathcal{J}^{+}$ be defined by \eqref{eq:JplusDef}.  Suppose $\Phi=\Phi(N') \ll T^{\varepsilon}$.
Then $\mathcal{J}^{+}(N') \ll   T^{\varepsilon} \cdot \mathcal{J}_0^{\times}(N') + O(T^{-100})$, where $\mathcal{J}_0^{\times}$ is defined by
\begin{equation}
\label{eq:JtimesDefNonOsc}
  \mathcal{J}_0^{\times}(N')
  = 
  \sum_{n_1} \frac{1}{n_1^2}
 \sum_{k' \asymp \frac{QT}{N}} \sum_{\substack{b \asymp \frac{Q}{n_1} \\ (b, k') = 1}}
 \Big| 
 \sum_{n_2 \asymp \frac{N'}{n_1^2}} 
\frac{\lambda_f(n_1, n_2)}{n_2}
S(\eta  \overline{k'}, n_2 ;b) 
w(\cdot)
\Big|^2 dy.
 \end{equation}
\end{myprop}
\begin{proof}
In this case, the exponential function is $T^{\varepsilon}$-inert, so it can be absorbed into the inert weight function.
\end{proof}

\subsection{Reduction to norms}
We define a norm
\begin{equation}
\mathcal{N}_1(N', Q, k, Y) = \max_{\| \alpha\| =1} 
%\sum_{n_1} \frac{1}{n_1^2}
 %\sum_{|k'| \asymp K } 
\sum_{\substack{b \asymp Q \\ (b,k) = 1}} 
  \int_{|y| \asymp Y}
  \Big| \sum_{\substack{n_2 \asymp N'}} \alpha_{n_2} 
   S(\eta \overline{k}, n_2 ;b) 
    n_2^{iy} \Big|^2 dy.
\end{equation}
Here, $\alpha_{n_2}$ are arbitrary complex numbers and $\| \alpha \| = \left(\sum_{n_2}\abs{\alpha_{n_2}}^2 \right)^\frac12$. 
We also define a version of this norm without the integral:
\begin{equation*}
\mathcal{N}_1(N', Q, k ) =  \max_{\| \alpha \| =1} 
%\sum_{n_1} \frac{1}{n_1^2}
%\sum_{|k'| \asymp K } 
\sum_{\substack{b \asymp Q \\ (b,k) = 1}} 
  \Big| \sum_{\substack{n_2 \asymp N'}} \alpha_{n_2} 
   S(\eta \overline{k}, n_2 ;b) 
    \Big|^2.
\end{equation*}

We define a slight variant of $\mathcal{N}_1$ by
\begin{equation}
\label{eq:N2normdef}
\mathcal{N}(N', Q, k, Y) = \max_{X \geq 1} X \cdot\mathcal{N}_1 \Big(\frac{N'}{X^2}, \frac{Q}{X}, k, Y\Big),
\end{equation}
and similarly 
\begin{equation}
\label{eq:N2normdefVariant}
\mathcal{N}(N', Q, k ) = \max_{X \geq 1} X \cdot \mathcal{N}_1 \Big(\frac{N'}{X^2}, \frac{Q}{X}, k\Big).
\end{equation}

\begin{mylemma}
\label{lemma:JboundInTermsOfNorm}
 Let $\Phi$ be as in \eqref{eq:PhiDef}.  If $\Phi \gg T^{\varepsilon}$, then
 \begin{equation*}
  \mathcal{J}^{\times}(N') \ll \frac{  T^{\varepsilon}}{N'} \sum_{k'\asymp\frac{QT}{N}}\mathcal{N}(N', Q, k', \Phi).
 \end{equation*}
 If $\Phi \ll T^{\varepsilon}$, then
 \begin{equation*}
  \mathcal{J}_0^{\times}(N') \ll \frac{  T^{\varepsilon}}{N'} \sum_{k'\asymp\frac{QT}{N}}\mathcal{N}(N', Q, k').
 \end{equation*}
 \end{mylemma}
 \begin{proof}
The proof for $\Phi \ll T^{\varepsilon}$ is easier, and the details are nearly identical, so we focus on $\Phi \gg T^{\varepsilon}$.
By the definitions of $\mathcal{J}^{\times}$ and $\mathcal{N}_1$, we have
 \begin{equation*}
  \mathcal{J}^{\times}(N') \ll 
  \sum_{n_1} \frac{1}{n_1^2}
\sum_{k'\asymp\frac{QT}{N}}\mathcal{N}_1\Big(\frac{N'}{n_1^2}, \frac{Q}{n_1}, k', \Phi\Big)
\cdot 
\sum_{n_2 \asymp \frac{N'}{n_1^2}} \frac{|\lambda_f(n_1, n_2)|^2}{  n_2^2}.
  \end{equation*}
  Writing this with the norm $\mathcal{N}$, we have 
   \begin{equation*}
  \mathcal{J}^{\times}(N') \ll 
\sum_{k'\asymp\frac{QT}{N}}\mathcal{N}(N', Q, k', \Phi)
\cdot 
 \sum_{n_1^2 n_2 \asymp N'} \frac{|\lambda_f(n_1, n_2)|^2}{n_1^3 n_2^2}.
 \end{equation*}
  Using Lemma \ref{lemma:RankinSelbergBound}  completes the proof.
  \end{proof}

Now we finally link the long second moment of interest with the norm $\mathcal{N}$ with the following.
\begin{myprop}
\label{prop:secondmomentVSnorm}
We have
\begin{multline}
\label{eq:secondmomentVSnorm}
\int_{T}^{2T} |L(f, 1/2+it)|^2 dt \ll 
T^\varepsilon
\max_{\eta_4 \in \{\pm 1\}}
\sum_{d}
\max_{\substack{Q \ll C/d \\ N' \ll \frac{N^2 T^{\varepsilon}}{d^3 C^3}}}
\sum_{k'\asymp\frac{QT}{N}}\Big(
\frac{1}{Q}
\mathcal{N}(N', Q, k', \Phi(N'))
\\
+\frac{1}{Q} \delta(\Phi(N') \ll T^\varepsilon)\mathcal{N}(N',Q,k')
\Big)
.
\end{multline}
\end{myprop}
\begin{proof}
   Recall Proposition \ref{prop:JwrtJplus} for  the bound on the long second moment in terms of the quantity $\mathcal{J}^+$. Propositions \ref{prop:JplusJtimesetc} and \ref{prop:JplusJtimesetcNonOsc} translate it to $\mathcal{J}^{\times}$. Then Lemma \ref{lemma:JboundInTermsOfNorm} bounds $\mathcal{J}^{\times}$ in terms of the norm $\mathcal{N}$.   Chaining these results together gives that the second moment is (essentially) bounded from above via
\begin{multline*}
   I_{0}(T)
\ll 
T^\varepsilon
\max_{\eta  \in \{\pm 1\}}
\sum_{d}
\max_{\substack{Q \ll C/d \\ N' \ll \frac{N^2 T^{\varepsilon}}{d^3 C^3}}}
\frac{ N^{1/3}}{Q^2} (N')^{1/3}
 \\
 \times \sum_{k'\asymp\frac{QT}{N}}\left(\Phi^{-1}\mathcal{N}(N',Q,k',\Phi)+\delta(\Phi\ll T^\varepsilon)\mathcal{N}(N',Q,k')\right),
\end{multline*}
where recall from \eqref{eq:PhiDef} that $\Phi = \frac{(N' N)^{1/3}}{Q}$.
Note that
\begin{equation*}
N
\Big(\frac{N'^2}{N Q^3}\Big)^{2/3} \frac{1}{\Phi N'}
\asymp Q^{-1}.
\end{equation*}
Hence
this bound simplifies as stated in the proposition. 
\end{proof}

\subsection{Variant norm and the large sieve}
\label{section:largesieveinterlude}
Define the following variant norm
\begin{equation*}
\mathcal{N}_2(N', Q, K, Y) = \max_{\|\alpha \| =1} 
%\sum_{n_1} \frac{1}{n_1^2}
 \sum_{|k| \asymp K } 
\sum_{\substack{b \asymp Q \\ (b,k) = 1}} 
  \int_{y \asymp Y}
  \Big| \sum_{\substack{n_2 \asymp N'}} \alpha_{n_2} 
   S(\eta \overline{k}, n_2 ;b) 
    n_2^{iy} \Big|^2 dy.
\end{equation*}
One can similarly define a norm without the integral, and also define the variant
$\mathcal{N}_3(N', Q, K, Y) = \max_{X \geq 1} X \cdot \mathcal{N}_2(N'/X^2, Q/X, K, Y)$, etc.  
Then we could just as well replace 
$\sum_{k'} \mathcal{N}(N', Q, k', \Phi)$ in
the right hand side of 
\eqref{eq:secondmomentVSnorm} with $\mathcal{N}(N', Q, K, \Phi)$ with $K= QT/N$.
Our motivation for studying $\mathcal{N}_1$ in place of $\mathcal{N}_2$ was justified at the end of Section \ref{subsec:Outline}.

To help gauge the depth of our forthcoming methods, we record here a straightforward bound on $\mathcal{N}_2$ following from the large sieve inequality.

\begin{mylemma}
\label{lemma:largesievenormbound}
Suppose $K \ll Q$.  Then
 \begin{equation}
 \label{eq:N2largesieve}
  \mathcal{N}_2(N', Q, K, Y) \ll Q (Q^2 Y + N').
 \end{equation}
The same bound holds for $\mathcal{N}_3(N', Q, K, Y)$.
\end{mylemma}
\begin{proof}
Since $K \ll Q$, we can extend the sum over $k$ to a complete period modulo $b$, and also drop the condition that $\overline{k}$ is coprime to $b$.  Then
\begin{equation*}
 \sum_{\substack{k \asymp K \\ (k, b) = 1}} \Big| \sum_{n} \alpha_n S(\eta \overline{k}, n;b) n^{iy} \Big|^2
 \ll \sum_{k \shortmod{b}} \Big| \sum_n \alpha_n S(\eta k, n;b) n^{iy} \Big|^2.
\end{equation*}
We have from orthogonality of characters that
\begin{align*}
 \sum_{k \shortmod{b}} S(\eta k, n;b) S(\eta k, n';b)
 % = 
 % \sum_{k \shortmod{b}}
 % \,\sumstar_{t,u \shortmod{b}} e_b(k \overline{t} + n t - k \overline{u} - n' u)  
= b \sumstar_{t \shortmod{b}} e_b((n-n') t).
 \end{align*}
Hence
\begin{equation*}
 \mathcal{N}_2(N', Q, K, Y) \ll  \max_{\|\alpha \| =1} 
\sum_{b \asymp Q} b 
 \sumstar_{t \shortmod{b}} 
  \int_{y \asymp Y}
  \Big| \sum_{\substack{n_2 \asymp N'}} \alpha_{n_2} 
   e_b(n_2 t)
    n_2^{iy} \Big|^2 dy.
\end{equation*}
The proof of \eqref{eq:N2largesieve} is completed using the hybrid large sieve inequality.
\end{proof}

Next we briefly check what Lemma \ref{lemma:largesievenormbound} implies for the long second moment.  
Ignoring minor technical details, 
we should obtain the bound
\begin{equation*}
    \int_T^{2T} |L(f, 1/2+it)|^2 dt
    \ll T^{\varepsilon} (Q^2 \Phi + N') \ll T^{\varepsilon} \Big(N + \frac{N^2}{C^3}\Big).
\end{equation*}
The term $N$ is the problem, as it only leads to a bound of the same quality as the mean value theorem for Dirichlet polynomials (as in \eqref{eq:MVTDPbound}).  
For what it is worth,
the other term $N^2/C^3$ would be acceptable for Theorem \ref{thm:mainthm}.

\subsection{Norm bound and deduction of Theorem \ref{thm:mainthm}}
The following theorem gives an improved estimate compared to Lemma \ref{lemma:largesievenormbound}.
\begin{mytheo}
\label{thm:dualnormbound}
 Suppose $k \ll Q$. Let $\mathrm{rad}(k)=\displaystyle\prod_{\substack{p|k\\\text{prime}}}p$. Then 
\begin{equation}
 \mathcal{N}(N', Q, k, Y) \ll T^{\varepsilon} Q\left( N' + \frac{k Q^2 Y^2}{\sqrt{\mathrm{rad}(k)N'}} + Q \sqrt{\frac{N'}{\mathrm{rad}(k)}} \right).
\end{equation}
In addition,
\begin{equation*}
\mathcal{N}(N', Q, k) \ll T^{\varepsilon} Q\left( N' + \frac{k Q^2}{\sqrt{\mathrm{rad}(k)N'}} + Q \sqrt{\frac{N'}{\mathrm{rad}(k)}} \right).
\end{equation*}
 \end{mytheo}
We will prove Theorem \ref{thm:dualnormbound} in Section \ref{sect:boundnorm}.
Taking Theorem \ref{thm:dualnormbound} for granted, we now complete the proof of Theorem \ref{thm:mainthm}.  
Applying the estimates from Theorem \ref{thm:dualnormbound} into \eqref{eq:secondmomentVSnorm} 
and using Lemma \ref{lemma:radicalsum}, we obtain
\begin{multline}
\label{eq:secondmomentAfterNormBound}
I_0(T) \ll T^\varepsilon
\max_{\eta \in \{\pm 1\}}
\sum_d
\max_{\substack{Q \ll C/d \\ N' \ll \frac{N^2 T^{\varepsilon}}{d^3 C^3}}}\sum_{k'\asymp\frac{QT}{N}} \left(N' + 
\frac{k'}{\sqrt{\mathrm{rad}(k')}}
\frac{ Q^2 \Phi^2}{\sqrt{N'}} + Q \sqrt{\frac{N'}{\mathrm{rad}(k')}}\right)\\
\ll T^\varepsilon\left(\frac{NT}{C^2}+\frac{CT^{3/2}}{\sqrt{N}}+\sqrt{NT}\right).
\end{multline}
We make the optimal choice $C = N^{1/2} T^{-1/6}$ (which is consistent with \eqref{eq:RsizeDef}) to give
\begin{equation*}
    \int_{T}^{2T} |L(f, 1/2+it)|^2 dt
    \ll T^{\varepsilon} \left( T^{4/3} +\sqrt{NT}\right).
\end{equation*}
Note the latter term is relatively smaller, taking the form
$\sqrt{NT} \ll T^{5/4+\varepsilon}$.

\section{Bounding the norm: proof of Theorem \ref{thm:dualnormbound}}\label{sect:boundnorm}
This section is devoted to proving Theorem \ref{thm:dualnormbound}.

\subsection{Duality}
Our approach to estimate $\mathcal{N}_1$ is to use duality.  The dual definition of $\mathcal{N}_1$ is
\begin{equation}
\mathcal{N}_1(N', Q, k, Y) = \max_{\|\beta\| =1} 
\sum_{\substack{n_2 \asymp N'}}
\Big|
\sum_{b \asymp Q} 
 %\sum_{|k'| \asymp K } 
  \int_{y \asymp Y}
  \beta(b,y) 
   S(\eta \overline{k}, n_2 ;b) 
    n_2^{iy} dy \Big|^2.
\end{equation}
Notice that $\beta$ may depend on $k$ and for convenience we assume that $\beta(b,y)$ is supported on integers with $(b, k) = 1$. We also assume $\eta = 1$, since the analysis of $\mathcal{N}_1$ from $\eta=-1$ is essentially the same.

\subsection{Poisson summation}
We derive a bound on $\mathcal{N}_1$ using Poisson summation. Opening the square we obtain
\begin{equation*}
\mathcal{N}_1(N', Q, k, Y)
\ll
\max_{ \|\beta\| =1}  
\sum_{\substack{b_1, b_2}}
\int_{y_1, y_2} \beta(b_1,  y_1)
\overline{\beta(b_2,  y_2)} 
\cdot
 S(\cdot) dy_1 dy_2,
\end{equation*}
where $S=S(\cdot) = S(b_1, b_2, k, y_1, y_2)$ is given by
\begin{equation*}
S = \sum_{n_2} w_{N'}(n_2) S(\overline{k}, n_2 ;b_1) S(\overline{k}, n_2 ;b_2 )
n_2^{iy_1 - i y_2}.
\end{equation*}
Applying Poisson summation modulo $b_1 b_2$ to $S$ gives
\begin{equation*}
S 
= \sum_{n'' \in \mz} \widehat{S} \cdot \widehat{I},
\end{equation*}
where
\begin{equation}
\label{eq:ShatDef}
\widehat{S} = 
\widehat{S}(b_1, b_2,k, n'') = 
\frac{1}{b_1 b_2} \sum_{x \shortmod{b_1 b_2}} S(\overline{k}, x ;b_1 ) S(\overline{k}, x ;b_2)
e_{b_1 b_2}(x n''),
\end{equation}
and
\begin{multline*}
\widehat{I} = \intR e\left(\frac{(y_1 -y_2)}{2 \pi} \log x  - \frac{x n''}{b_1 b_2}\right) w_{N'}( x) dx
\\
= (N')^{1+ iy_1 - i y_2} \intR e\left(\frac{(y_1 -y_2)}{2 \pi} \log x  - \frac{x n'' N'}{b_1 b_2}\right) w_{N'}(N' x) dx.
\end{multline*}
Let $\mathcal{N}_1 \leq \mathcal{N}_1^{(0)} + \sum_{\pm} |\mathcal{N}_1^{\pm}|$, where $\mathcal{N}_1^{(0)}$ consists of the term $n'' = 0$, and where $\mathcal{N}_1^{\pm}$ consists of the terms with $\pm n'' > 0$.  We will treat $\mathcal{N}_1^{(0)}$ and $\mathcal{N}_1^{+}$, since $\mathcal{N}_1^{-}$ is similar to the latter case.

It is convenient to apply a (partial) dyadic partition of unity of the form
\begin{equation*}
1 = \sum_{U} \omega\left(\frac{y_1 - y_2}{U}\right)=\sum_{Y^\varepsilon\ll U\ll Y} \omega\left(\frac{y_1 - y_2}{U}\right)+\omega_0\left(\frac{y_1-y_2}{Y^\varepsilon}\right),
\end{equation*}
where $\omega_0$ is some fixed smooth function supported on $[-1,1]$.
Let $\widehat{I}(U)$ and $\widehat{I}(1)$ be the portion of $\widehat{I}$ with $|y_1 - y_2| \asymp U$ and $|y_1-y_2|\ll Y^\varepsilon$ respectively.
Integration by parts shows $\widehat{I}(U)$ and $\widehat{I}(1)$ are very small except for
\begin{equation}
n'' \ll N'' := \frac{Q^2}{N'} (Y^{\varepsilon} + U).
\end{equation}
In addition, we have the essentially best-possible bound of
\begin{equation*}
\widehat{I}(1) \ll N'.
\end{equation*}
On the other hand for $U \gg Y^{\varepsilon}$, and $n'' > 0$ (as required in $\mathcal{N}_1^{+}$)
by Proposition \ref{prop:statphase} (stationary phase)
we have
\begin{equation}
\label{eq:IofUapproximation}
\widehat{I}(U) = \frac{N'}{\sqrt{U}} \Big(\frac{b_1 b_2 (y_1 -y_2)}{2 \pi e n''  } \Big)^{i(y_1 - y_2)} 
%\omega\left(\frac{y_1 - y_2}{U}\right)
w
(\cdot) + O(T^{-2024}),
\end{equation}
for some inert function $w$ supported on $y_1-y_2\asymp U$ and 
\begin{equation}
\label{eq:n''size}
n'' \asymp \frac{Q^2 U}{N'}.
\end{equation}

The trickier part is to evaluate $\widehat{S}$.  As a partial evaluation, we have the following.
\begin{mylemma}
\label{lemma:ShatEvaluationPartial}
  Suppose $b_1 = g_1 b_1'$ and $b_2 = g_2 b_2'$ where $(b_1' b_2', g_1 g_2) = (b_1', b_2') = 1$, and $g_1$ and $g_2$ share all the same prime factors.  Let $g_0 = \frac{g_1g_2}{(g_1,g_2)}$ %be the lcm of $g_1$ and $g_2$, 
  and let $g_i' = \frac{g_i}{(g_1, g_2)}$. Then $\widehat{S} = 0$ unless $(g_1, g_2) | n''$.  Also, $\widehat{S} = 0$ unless $(n'', b_1' b_2') = 1$.
If $n'' = (g_1, g_2) m$, then
$\widehat{S}$ is given by
\begin{equation}
 \widehat{S} = 
 e_{b_1'}(-b_2' g_2' \overline{g_1 m k})
 e_{b_2'}(-b_1' g_1' \overline{g_2 m k})
 \mathop{
 \sumstar_{t_1 \shortmod{g_1}} \thinspace  \sumstar_{t_2 \shortmod{g_2}}}_{g_2' t_1 + g_1' t_2 + m \equiv 0 \shortmod{g_0}}
 e_{g_1}((b_2' \overline{b_1'}) \overline{t_1 k})
 e_{g_2}((b_1' \overline{b_2'}) \overline{t_2 k}).
\end{equation}
\end{mylemma}
In subsequent developments, 
we will write $g_1 \sim g_2$ to mean that $g_1$ and $g_2$ share all the same prime factors.

\begin{proof}
  Opening the definitions of the Kloosterman sums in \eqref{eq:ShatDef} and summing over $x$ using orthogonality of characters, we obtain
 \begin{equation*}
  \widehat{S} = 
\mathop{\sumstar_{t_1 \shortmod{b_1}} \thinspace \sumstar_{t_2 \shortmod{b_2}}}_{b_2 t_1 + b_1 t_2 + n'' \equiv 0 \shortmod{b_1 b_2}}
e_{b_1}(\overline{t_1 k}) e_{b_2}(\overline{t_2 k}).
 \end{equation*}
By the Chinese remainder theorem, the linear congruence is equivalent to a system of congruences of moduli $b_1'$, $b_2'$, and $g_1 g_2$.  Modulo $b_1'$, it simplifies to read $b_2 t_1 + n'' \equiv 0 \pmod{b_1'}$, i.e., $t_1 \equiv - \overline{b_2} n'' \pmod{b_1'}$.  Similarly, we obtain $t_2 \equiv - \overline{b_1} n'' \pmod{b_2'}$.  
Note that there are no solutions to these congruences unless $(n'', b_1' b_2') = 1$.
 Hence
 \begin{equation*}
  \widehat{S} = 
e_{b_1'}(-b_2' g_2 \overline{g_1 n'' k}) e_{b_2'}(-  b_1' g_1 \overline{g_2 n'' k}) 
  \mathop{\sumstar_{t_1 \shortmod{g_1}} \thinspace \sumstar_{t_2 \shortmod{g_2}}}_{b_2' g_2 t_1 + b_1' g_1 t_2 + n'' \equiv 0 \shortmod{g_1 g_2}}
e_{g_1}(\overline{b_1' t_1 k}) 
e_{g_2}(\overline{b_2' t_2 k}).
 \end{equation*}
In addition, we note that the congruence implies $(g_1, g_2) | n''$.  Letting $m=\frac{n''}{(g_1, g_2)}$ we get
 \begin{equation*}
  e_{b_1'}(-b_2' g_2 \overline{g_1 n'' k})
  = e_{b_1'}(-b_2' g_2' \overline{g_1 m k}),
 \end{equation*}
and a similar equation holds for the other exponential factor, of modulus $b_2'$.  For the remaining sum, we change variables $t_1 \rightarrow (b_2')^{-1} t_1$ and $t_2 \rightarrow (b_1')^{-1} t_2$.  The congruence  is then equivalent to $g_2' t_1 + g_1' t_2 + m \equiv 0 \pmod{g_0}$.
This gives the result.
\end{proof}

\subsection{The diagonal term}
\begin{mylemma}\label{lemma:Diagonal}
We have
\begin{equation*}
\mathcal{N}_1^{(0)}(N', Q, k, Y) \ll 
Q N' Y^{\varepsilon}.
\end{equation*}
\end{mylemma}
\begin{proof}
 From Lemma \ref{lemma:ShatEvaluationPartial} we have that $n'' = 0$ implies that $b_1' = b_2' = 1$.  The congruence $g_2' t_1 + g_1' t_2 \equiv 0 \pmod{g_0}$ then implies that $g_1 = g_2$, since $(t_1, g_1) = (t_2, g_2) = 1$.  Then
\begin{equation*}
\widehat{S}(g, g, k, 0) = S(0, 0 ; g)=\phi(g)\ll Q.
\end{equation*}
From the earlier discussion on $\widehat{I}$ we may assume $|y_1 - y_2| \ll Y^{\varepsilon}$.  
Hence the contribution to $\mathcal{N}_1$ from $n''=0$ is bounded by
\begin{equation*}
QN' \max_{ \|\beta\| =1}  
\sum_{\substack{g}}
\int_{|y_1 - y_2| \ll Y^{\varepsilon}} \beta(g,  y_1)
\overline{\beta(g,  y_2)}
 dy_1 dy_2.
\end{equation*}
A simple use of the AM-GM inequality bounds this in turn by
\begin{equation*}
QN' Y^{\varepsilon} \max_{ \|\beta\| =1}  
\sum_{g} \int_{y}
|\beta(g,  y)|^2
 dy =  QN'Y^\varepsilon. \qedhere
\end{equation*}
\end{proof}

\subsection{Reciprocity}
Suppose $m \geq 1$ and $k \geq 1$.
Applying reciprocity gives 
\begin{equation*}
e_{b_1'}(-b_2' g_2' \overline{g_1 m k})
 e_{b_2'}(-b_1' g_1' \overline{g_2 m k})
 = 
 e_{g_1 m k}( b_2' \overline{b_1'} g_2')
 e_{g_2 m k}(b_1' \overline{b_2'} g_1')
 e\Big(\frac{-b_2' g_2'}{b_1' g_1 m k}\Big) e\Big(\frac{-b_1' g_1'}{b_2' g_2 m k}\Big).
\end{equation*}
Note the argument of one of the above exponential factors satisfies
\begin{equation*}
 \frac{b_2' g_2'}{b_1' g_1 m k}
 = \frac{b_2}{b_1  n'' k} \asymp \frac{1}{n'' k}  \ll 1,
\end{equation*}
and a similar bound holds for the other factor.  Hence these correction factors may be absorbed into the inert weight function.
For $x \in \mz$ with $\gcd(x, g_1 g_2 k m) = 1$, let
\begin{equation*}
 G(x, g_1, g_2, k, m)
 = e_{g_1 m k}(  \overline{x} g_2')
 e_{g_2 m k}(x g_1')
 \mathop{
 \sumstar_{t_1 \shortmod{g_1}} \thinspace  \sumstar_{t_2 \shortmod{g_2}}}_{g_2' t_1 + g_1' t_2 + m \equiv 0 \shortmod{g_0}}
 e_{g_1}( \overline{x t_1 k})
 e_{g_2}(x \overline{t_2 k}).
\end{equation*}
Then, (mixing some notation), we have
\begin{equation*}
 \widehat{S}(b_1, b_2, k, n'')
 = G(b_1' \overline{b_2'}, g_1, g_2, k, m) \cdot w(\cdot).
\end{equation*}
More generally, for
$A_1, A_2 \in \mz$ with
$(A_1, g_1 k m) = (A_2, g_2 k m)=1$, we define
\begin{multline}
\label{eq:GA1A2def}
 G(x, g_1, g_2, k, m, A_1, A_2) %, A_3, A_4)
 \\
 = e_{g_1 m k}(  \overline{x} g_2' A_1)
 e_{g_2 m k}(x g_1' A_2)
 \mathop{
 \sumstar_{t_1 \shortmod{g_1}} \thinspace  \sumstar_{t_2 \shortmod{g_2}}}_{g_2' t_1 + g_1' t_2 + m \equiv 0 \shortmod{g_0}}
 e_{g_1}( \overline{x t_1 k} A_1)
 e_{g_2}(x \overline{t_2 k} A_2).
\end{multline}
Then $G(x, g_1, g_2, k, m) = G(x, g_1, g_2, k, m, 1, 1)$.

Summarizing the developments so far, we have
\begin{multline}
\label{eq:N1+ViaG}
\mathcal{N}_1^{+}(N', Q, k, Y)
\ll
\max_{ \|\beta\| =1}  
\max_{\substack{1 \ll U\ll Y}} \frac{N'}{\sqrt{U}}
\Big|
\mathop{\sum_{g_1, g_2}}_{g_1 \sim g_2}
\sum_{\substack{(b_1', b_2') = 1 \\ (b_1' b_2', g_1 g_2) = 1\\(b_1'b_2'g_1g_2,k)=1)}}
\sum_{m > 0} 
\int_{y_1, y_2} \beta(b_1' g_1,  y_1)
\overline{\beta(b_2' g_2,  y_2)} 
\\
G(b_1' \overline{b_2'}, g_1, g_2, k, m)
\Big(\frac{g_1 b_1' g_2 b_2' (y_1 -y_2)}{(g_1, g_2) m  \cdot 2 \pi e} \Big)^{i(y_1 - y_2)} 
\omega_U(y_1 - y_2)
w
(\cdot)
 dy_1 dy_2 \Big|.
\end{multline}
Here $\omega_U(x)=\omega(x/U)$ for $U\gg Y^\varepsilon$ and $w$ is some $T^\varepsilon$-inert function as before.

\subsection{Character sum bounds}
\begin{mylemma}
\label{lemma:GtwistedMultiplicative}
 Suppose for $i=1,2$, $g_i = g_i^{(1)} g_i^{(2)}$, $g_i'^{(1)}=g_i^{(1)}/(g_1^{(1)},g_2^{(1)})$, $g_i'^{(2)}=g_i^{(2)}/(g_1^{(2)},g_2^{(2)})$, $k = k^{(1)} k^{(2)}$, and $m = m^{(1)} m^{(2)}$.  For $i=1,2$, let $q^{(i)} = g_1^{(i)} g_2^{(i)} k^{(i)} m^{(i)}$, and suppose $(q^{(1)}, q^{(2)}) = 1$.  
 Let $A_1^{(1)} = A_1 \cdot(g_2')^{(2)}  \overline{g_1^{(2)} m^{(2)} k^{(2)}}$,
 $A_2^{(1)} =  A_2 \cdot  (g_1')^{(2)} \overline{g_2^{(2)} m^{(2)} k^{(2)}}$,
 $A_1^{(2)} = A_1 \cdot   (g_2')^{(1)} \overline{g_1^{(1)} m^{(1)} k^{(1)}} $, and $A_2^{(2)} = A_2 \cdot   (g_1')^{(1)} \overline{g_2^{(1)} m^{(1)} k^{(1)}}$.
 Then 
 \begin{equation*}
  G(x, g_1, g_2, k, m, A_1, A_2)
  = 
  \prod_{i=1}^{2} G(x, g_1^{(i)}, g_2^{(i)}, k^{(i)}, m^{(i)}, A_1^{(i)}, A_2^{(i)}).
 \end{equation*}
 \end{mylemma}
 
\begin{proof}
 This follows from the Chinese remainder theorem, though some points are worth discussing.  
 For
the additive character aspect, we have that the $(1)$-part of $e_{g_1 m k}(\overline{x} g_2' A_1)$ equals
 \begin{equation*}
 e_{g_1^{(1)}  m^{(1)} k^{(1)}} 
  (\overline{x} g_2' \overline{g_1^{(2)} m^{(2)} k^{(2)}}  A_1)
  =
  e_{g_1^{(1)}  m^{(1)} k^{(1)}} 
  (\overline{x} (g_2')^{(1)} \cdot \underbrace{(g_2')^{(2)} \overline{g_1^{(2)} m^{(2)} k^{(2)}}   A_1}_{\text{$A_1^{(1)}$}}).
 \end{equation*}
The other parts of these additive character components are similar.

 The sum over $t_1$ and $t_2$ modulo $g_1$ and $g_2$  factors by the CRT. Let $g_0^{(1)} = \frac{g_1^{(1)} g_2^{(1)}}{(g_1^{(1)}, g_2^{(1)})}$, and let $$F^{(1)}(t_1,t_2) =(g_2')^{(1)} (g_2')^{(2)} t_1 + (g_1')^{(1)} (g_1')^{(2)} t_2 + m^{(1)} m^{(2)}.$$ Then, the $(1)$-part of the sum takes the form
 \begin{equation*}
  \mathop{
 \sumstar_{t_1 \shortmod{g_1^{(1)}}} \thinspace  \sumstar_{t_2 \shortmod{g_2^{(1)}}}}_{F^{(1)}(t_1,t_2) \equiv 0 \shortmod{g_0^{(1)}}} 
 e_{g_1^{(1)}}(\overline{xg_1^{(2)}t_1 k^{(1)} k^{(2)}} A_1)
 e_{g_2^{(1)}}(x\overline{g_2^{(2)}t_2 k^{(1)} k^{(2)}} A_2).
 \end{equation*}
Changing variables $t_1 \rightarrow \overline{(g_2')^{(2)}} m^{(2)} t_1$ and $t_2 \rightarrow \overline{(g_1')^{(2)}} m^{(2)} t_2$ shows that the sum has the claimed format, which completes the proof.
 \end{proof}

Next we change basis to multiplicative functions.  With $q = q^{(1)} q^{(2)}$ as in Lemma \ref{lemma:GtwistedMultiplicative}, we have that $G(x, \cdot)$ is a periodic function in $x$ modulo $q$, with $\gcd(x,q) = 1$.  Then by multiplicative Fourier inversion, we have
\begin{equation}
\label{eq:GtoGhat}
 G(x, \cdot) = \sum_{\chi \shortmod{q}} \widehat{G}(\chi, \cdot) \chi(x),
\end{equation}
where
\begin{equation}
\label{eq:Ghatdef}
 \widehat{G}(\chi, \cdot)
 = \widehat{G}(\chi, g_1, g_2, k, m, A_1, A_2)
 =
 \frac{1}{\varphi(q)} 
 \sum_{y \shortmod{q}} G(y, g_1, g_2, k, m, A_1, A_2) \overline{\chi}(y).
\end{equation}

Then we define
\begin{equation}
 H(\chi, g_1, g_2, k, m)
 = \max_{A_i} |\widehat{G}(\chi, g_1, g_2, k, m, A_1, A_2)|,
\end{equation}
where the maximum is over integers $A_i$, coprime to $g_i km$.  Now the twisted multiplicativity of $G$ from Lemma \ref{lemma:GtwistedMultiplicative} implies that $H$ is multiplicative.  That is, if we write $\chi = \chi^{(1)} \chi^{(2)}$ with $\chi^{(i)}$ modulo $q^{(i)}$ then
\begin{equation*}
H(\chi, g_1, g_2, k, m) =
\prod_{i=1}^{2} H(\chi^{(i)}, g_1^{(i)}, g_2^{(i)}, k^{(i)}, m^{(i)}).
\end{equation*}
Hence to bound $H$, it suffices to consider one prime at a time.

We record some assumptions that we may impose in our analysis: $(g_1g_2, k) = 1$, and $g_1 \sim g_2$. 
\begin{mylemma}
\label{lemma:Hbound}
Let notation and conditions be as above.
Suppose $p$ is a prime and write the variables additively, replacing $g_1$ by $p^{\gamma_1}$, etc.  Write $e(\chi)$ for the conductor exponent of $\chi$.
\begin{enumerate}
\item Suppose that $\gamma_1 \neq \gamma_2 > 0$.  Then $H(\chi, p^{\gamma_1}, p^{\gamma_2}, 1, p^\mu) = 0$ unless $\mu=0$ and $\chi$ is trivial, in which case $H(\chi_0, p^{\gamma_1}, p^{\gamma_2}, 1, 1) \leq 2 p^{\min(\gamma_1, \gamma_2)/2}$.
\item Suppose that $\gamma_1 = \gamma_2 = \gamma > 0$.  Then $H(\chi, p^{\gamma}, p^{\gamma}, 1, p^\mu) = 0$ unless $e(\chi) \leq \mu + 1$.  
\begin{itemize}
    \item [(a)] Additionally  if $e(\chi) \leq \mu $, then
$H(\chi, p^\gamma, p^\gamma, 1, p^\mu) \leq 4p^{\frac{\gamma-\mu}{2}}.$ 
    \item [(b)] If $e(\chi)=\mu+1$, and $\gamma > 1$, then $H(\chi, p^{\gamma}, p^{\gamma}, 1, p^\mu) \leq 4\frac{p^{\frac{\gamma-\mu}{2}}}{p-1}$.   
\end{itemize}
 \item  If $\gamma_1 = \gamma_2 = 1$, and $e(\chi)=\mu+1$, we have $H(\chi, p, p, 1, p^\mu) \leq 2 \frac{p^{\frac12}}{p-1}$.
\item In all the remaining cases we have $\gamma_1=\gamma_2=0$ and
    $H(\chi, 1, 1, p^{\kappa}, p^\mu) \ll p^{-1/2}$.
\end{enumerate}
\end{mylemma}
\begin{proof}
For convenience we work with multiplicative notation, but keeping the assumption that all variables are powers of $p$.  
We have
by the definitions \eqref{eq:Ghatdef} and 
\eqref{eq:GA1A2def} that $\widehat{G}(\chi, g_1, g_2, k, m, A_1, A_2)$ equals
\begin{equation}
\label{eq:GhatWrittenOut}
\frac{1}{\varphi(q)} 
 \mathop{
 \sumstar_{t_1 \shortmod{g_1}} \thinspace  \sumstar_{t_2 \shortmod{g_2}}}_{g_2' t_1 + g_1' t_2 + m \equiv 0 \shortmod{g_0}}
  \sum_{y \shortmod{q}} \overline{\chi}(y)
 e_{g_1 m k}(  \overline{y} A_1 (g_2' + mk \overline{t_1 k}))
 e_{g_2 m k}(y  A_2 (g_1' + mk \overline{t_2 k}))).
\end{equation}
In the above $\overline{k}$ is the multiplicative inverse of $k$ modulo $g_1 g_2$.

By symmetry, say $g_2 | g_1$, so that $g_2' = 1$, $g_1' = \frac{g_1}{g_2}$, and $g_0 = g_1$. 
The congruence in the sum implies $g_1' t_2 + m$ is coprime to $p$.  Hence if $p|g_1'$ then $p \nmid m$, and similarly if $p|m$ then $p \nmid g_1'$ (that is, $g_1' = 1$).  Therefore, we may assume $(g_1' + mk \overline{k t_2}, p) = 1$.
We will then change variables $y \rightarrow y (g_1'+mk \overline{k} \overline{t_2})^{-1}$.  To aid in the resulting simplification, note that, using the congruence $t_1 + g_1' t_2 + m \equiv 0 \pmod{g_1}$ implies
\begin{equation*}
    (1+mk \overline{t_1 k})(g_1' + mk \overline{t_2 k})
    \equiv g_1' + mk \overline{k t_1 t_2} (t_1 + g_1' t_2 + mk \overline{k}) 
    %\pmod{g_1 mk}
    \equiv g_1' \pmod{g_1 mk}.
\end{equation*}
Hence \eqref{eq:GhatWrittenOut} simplifies as
\begin{equation}
\label{eq:GhatWrittenOut2}
\frac{1}{\varphi(g_1 mk)} 
 \mathop{
 \sumstar_{t_1 \shortmod{g_1}} \thinspace  \sumstar_{t_2 \shortmod{g_2}}}_{ t_1 + g_1' t_2 + m \equiv 0 \shortmod{g_0}}
 \chi(g_1'+mk \overline{k t_2})
  \sum_{y \shortmod{g_1 mk}} \overline{\chi}(y) 
 e_{g_1 m k}(  \overline{y} A_1 g_1')
 e_{g_2 m k}(y  A_2).
\end{equation}

 {\bf Case (1)}. In this case, $\kappa=0$ ($k=1$) and $p|g_1'$.     
The congruence in \eqref{eq:GhatWrittenOut2} implies that the sum vanishes unless $\mu = 0$.
Hence \eqref{eq:GhatWrittenOut2} simplifies as
\begin{equation*}
\Big(
\mathop{\sumstar_{t_1 \shortmod{g_1}} \thinspace  \sumstar_{t_2 \shortmod{g_2}}}_{  t_1 + g_1' t_2 + 1 \equiv 0 \shortmod{g_1}}
\chi(g_1' +\overline{ t_2})
\Big)
\Big( 
 \frac{1}{\varphi(g_1)} 
 \sum_{y \shortmod{g_1}} \overline{\chi}(y)
 e_{g_1}(\overline{y} A_1 g_1')
 e_{g_2}(y A_2)\Big).
\end{equation*}
Note the factorization on display.  For each $t_2$, the congruence uniquely determines $t_1$.  Hence the sum over $t_1$ and $t_2$ in parentheses simplifies as
\begin{equation*}
\sumstar_{t_2 \shortmod{g_2}} \chi(g_1' + t_2) 
= \sum_{t_2 \shortmod{g_2}} \chi(g_1' + t_2)
= \sum_{t_2 \shortmod{g_2}} \chi(t_2),
\end{equation*}
where we used that $p | g_1'$ to relax the condition that $t_2$ is coprime to $g_2$. Now the sum over $t_2$ vanishes unless $\chi$ is trivial.  Next observe $g_1'/g_1 = 1/g_2$, so the sum over $y$ is seen as a Kloosterman sum repeated $g_1/g_2$ times, giving
\begin{equation*}
\widehat{G}(\chi_0, g_1, g_2, 1, 1, A_1, A_2)
= S(A_1, A_2, g_2).
\end{equation*} 
The Weil bound finishes this case.  

{\bf Case (2)}.  Here $g_1' = g_2' = 1$, and again $k =1$.  
Then
\eqref{eq:GhatWrittenOut2} 
simplifies as
\begin{equation*}
\Big(
\mathop{
 \sumstar_{t_1 \shortmod{g}} \thinspace  \sumstar_{t_2 \shortmod{g}}}_{ t_1 +  t_2 + m \equiv 0 \shortmod{g}}
\chi(1 + m \overline{t_1})
\Big)
\Big( 
 \frac{1}{\varphi(gm)}
\sum_{y \shortmod{gm}} \overline{\chi}(y) 
e_{g m}(  \overline{y}  A_1 )
 e_{g m}(y A_2) \Big).
\end{equation*}
 Given any $t_1$ coprime to $g$, there is a unique $t_2$ (coprime to $g$) solving the congruence.  Hence the sum over $t_1$ and $t_2$ simplifies as
 \begin{equation*}
 \sumstar_{t_1 \shortmod{g}} \chi(1 + m \overline{t_1})
 = \sumstar_{t_1 \shortmod{g}} \chi(1+ m t_1).
 \end{equation*}
Now we switch temporarily to additive notation.  This sum over $t_1$ becomes
\begin{equation}
\label{eq:charactersumsimp}
\sum_{t \shortmod{p^{\gamma}}} \chi(1 + p^{\mu} t)
- \sum_{t \shortmod{p^{\gamma-1}}} \chi(1 + p^{\mu+1} t).
\end{equation}
If $\mu>0$, \eqref{eq:charactersumsimp} simplifies as $p^{\gamma} \delta_{e(\chi) \leq \mu}
 - p^{\gamma-1} \delta_{e(\chi) \leq \mu+1}$. 
If $\mu=0$, then instead \eqref{eq:charactersumsimp} evaluates to $(p^{\gamma}-p^{\gamma-1}) \delta_{e(\chi)=0}  - p^{\gamma-1} \delta_{e(\chi) \leq 1}.$
Hence if $e(\chi) > \mu + 1$ we have $H(\chi, p^{\gamma}, p^{\gamma}, 1, p^\mu) = 0$. 
Overall, the sum in \eqref{eq:charactersumsimp} is bounded by
\begin{equation*}
\begin{cases}
0, \qquad & e(\chi) > \mu+1 \\
p^{\gamma-1}, \qquad & e(\chi) = \mu+1 \\
\varphi(p^{\gamma}), \qquad & e(\chi) \leq \mu.
\end{cases}
\end{equation*}

Case 2(a).  If $e(\chi) \leq \mu$,
\begin{equation*}
    |\widehat{G}(\chi, \cdot)| \leq p^{-\mu} |S_{\overline{\chi}}(A_1, A_2;p^{\gamma+\mu})|.
\end{equation*}
Note that as $\gamma >0$ and $e(\chi) \leq \mu$, we have $\mathrm{cond}(\chi) \leq p^{\gamma+\mu-1}$. Now, if $\gamma+\mu \geq 2$, we can use Lemma \ref{lemma:WeilTwisted} (2) and (4) to get the required bound. Otherwise we must have $\gamma=1, \mu =0$, in which case the bound in Lemma \ref{lemma:WeilTwisted} (1) suffices. This proves the case 2(a).

If $e(\chi) = \mu+1$ ,  
\begin{equation}
    \label{eq:echimuplus1}
    |\widehat{G}(\chi, \cdot)| \leq \frac{p^{-\mu}}{(p-1) } |S_{\overline{\chi}}(A_1, A_2;p^{\gamma+\mu})|.
\end{equation}
Case 2(b).
Once again we use Lemma \ref{lemma:WeilTwisted} (2) and (4) to get the required bounds on \eqref{eq:echimuplus1} (when $\gamma > 1$).

{\bf Case (3)}. We use Lemma \ref{lemma:WeilTwisted} (3) to bound \eqref{eq:echimuplus1} when $\gamma=1$.

{\bf Case (4)}.  We have 
directly from the definition \eqref{eq:GhatWrittenOut} that
\begin{equation*}
\widehat{G}(\chi, \cdot)
= \frac{p^{-\mu-\kappa}}{1-p^{-1}} 
S_{\overline{\chi}}(A_1, A_2 ; p^{\mu+\kappa}).
\end{equation*}
By Lemma \ref{lemma:WeilTwisted} (3), the proof is complete.
\end{proof}

\begin{myremark}
    We note that for the proof of Lemma \ref{lemma:Hbound}, we used square root cancellation of the twisted Kloosterman sum in cases $1, 2(a), 2(b)$, and a (possibly smaller) saving of $\sqrt{p}$ in cases $(3)$ and $(4)$. This distinction will be used to split primes in the next section.
\end{myremark}

\subsection{Final steps}
We pick up from \eqref{eq:N1+ViaG}
and apply \eqref{eq:GtoGhat}, giving
\begin{multline*}
\mathcal{N}_1^{+}(N', Q, k, Y)
\ll
\max_{ \|\beta\| =1}  
\max_{\substack{1 \ll U\ll Y}}
\Big|
\frac{N'}{\sqrt{U}}
\sum_{g_1 \sim g_2}
\sum_{\substack{(b_1', b_2') = 1 \\ (b_1' b_2', g_1 g_2) = 1\\(b_1'b_2'g_1g_2,k)=1)}}
\sum_{m > 0} 
\int_{y_1, y_2} \beta(b_1' g_1,  y_1)
\overline{\beta(b_2' g_2,  y_2)} 
\\
\sum_{\chi \shortmod{g_1 g_2 m k}}
\widehat{G}(\chi, g_1, g_2, k, m)
\Big(\frac{g_1 b_1' g_2 b_2' (y_1 -y_2)}{2 \pi e (g_1, g_2) m  } \Big)^{i(y_1 - y_2)} 
\chi(b_1' \overline{b_2'})
\omega_U(y_1 - y_2)
w
(\cdot)
 dy_1 dy_2 \Big|.
\end{multline*}
Now we apply M\"{o}bius inversion to remove the condition $(b_1', b_2') = 1$, rearrange, 
and use $|\widehat{G}(\chi, \cdot)| \leq H(\chi, \cdot)$,
giving
\begin{multline}
 \label{eq:N1plusBilinearBound}
\mathcal{N}_1^{+}(N', Q, k, Y)
\ll
\max_{ \|\beta\| =1}  
\max_{\substack{1 \ll U\ll Y}}
\frac{N'}{\sqrt{U}}
\sum_{g_1 \sim g_2}
\sum_{(a,g_1g_2k)=1}
\\
\times \int_{y_1, y_2} \omega_U(y_1 -y_2) 
\sum_{m > 0} 
\sum_{\chi \shortmod{g_1 g_2 m k}} 
H(\chi, g_1, g_2, k, m)
\cdot 
| \mathcal{C}_1| \cdot | \mathcal{C}_2| dy_1 dy_2,
\end{multline}
where
\begin{equation*}
\mathcal{C}_1
= \sum_{\substack{ (b_1', g_1 g_2 k) = 1}}
\beta(a b_1' g_1,  y_1) (b_1')^{i(y_1 -y_2)} \chi(b_1'),
\end{equation*}
and
\begin{equation*}
\mathcal{C}_2
= \sum_{\substack{ (b_2', g_1 g_2 k) = 1}}
\beta(a b_2' g_2,  y_2) (b_2')^{i(y_1 -y_2)} \overline{\chi}(b_2').
\end{equation*}
In order to simplify this expression, we need bounds on $H$, which are recorded locally (prime-by-prime) in Lemma \ref{lemma:Hbound}.  We enumerate the types of primes as follows:
\begin{enumerate}[label=\Roman*.]
\item \begin{itemize}
        \item[(a)]  $\gamma_1 \neq \gamma_2 \geq 1$, (hence necessarily $\mu=0, \ \chi = \chi_0$),
        \item[(b)]  $\gamma_1 = \gamma_2 = \gamma \geq 1, \ e(\chi) \leq \mu$,
    \end{itemize} 
\item $\gamma_1 = \gamma_2 = \gamma > 1, \ e(\chi) = \mu+1$,
\item \begin{itemize}
        \item[(a)]  $\gamma_1 = \gamma_2 = 1, \ e(\chi) = \mu+1$,
        \item[(b)] $\gamma_1=\gamma_2=0, \ \kappa \geq 0, \ \mu\geq 0$.
    \end{itemize}
\end{enumerate}
Lemma \ref{lemma:Hbound} implies that every prime may be assumed to fall into Type I, II or III. For example, a case not appearing in the list is $\gamma_1=\gamma_2\geq2$, $\mu=0$, $e(\chi) \geq 2$, but this case has $H=0$. Note that Type I and II correspond to cases (1) and (2) in Lemma \ref{lemma:Hbound}.

Next we factor each of the variables $g_1$, $g_2$, and $m$ according to the types of primes.  That is, we write $g_1 = \prod_{i=1}^3 g_{1,i}$, 
$g_2= \prod_{i=1}^{3} g_{2,i}$, 
and $m = \prod_{i=1}^3 m_i$, where the $i$-th factor corresponds to the product of primes of type $i$ from the above list. Observe $k$ would equal $k_3$ so there was no need to factor it.
We correspondingly factor $\chi = \prod_{i=1}^3 \chi_i$.  Next we record information about $H$ in these cases, which is simply translating the bounds from Lemma \ref{lemma:Hbound} into the above (global) notation.
Below we use the definition $\mathrm{rad}(x)=\displaystyle\prod_{\substack{p|x \\ \text{prime}}}p$.
\begin{enumerate}[label=\Roman*.]
\item $\cond(\chi_1) \mid m_1$, and
$H(\chi_1, g_{1,1}, g_{2,1}, 1, m_1) \ll \sqrt{\frac{(g_{1,1}, g_{2,1})}{m_1}} ((g_{1,1}, g_{2,1})m_1)^{\varepsilon}$.
\item $g_{1,2} = g_{2,2}$ squarefull, $\cond(\chi_2) = m_2 \cdot \mathrm{rad}(g_{1,2})$, and 
$$H(\chi_2, g_{1,2}, g_{1,2}, 1, m_2) \ll \frac{1}{\mathrm{rad}(g_{1,2})}\sqrt{\frac{g_{1,2}}{m_{2}}} (g_{1,2} m_2)^{\varepsilon}.$$
\item $g_{1,3} = g_{2,3} =: g_3$ squarefree, $\cond(\chi_3) \mid m_3 g_3 k$, $(g_3,k)=1$ and
 $$H(\chi_3, g_3,g_3,k,m_3) \ll \mathrm{rad}(m_3 g_3 k)^{-1/2 + \varepsilon}. $$
\end{enumerate}

We return to \eqref{eq:N1plusBilinearBound}, with these factorizations. 
By the AM-GM inequality, we have
\begin{multline*}
    \mathcal{N}_1^{+}(N', Q, k, Y) \ll \max_{ \|\beta\| =1}  
\max_{\substack{1 \ll U\ll Y}} \frac{N'}{\sqrt{U}}
\sum_{g_1 \sim g_2}
\sum_{(a,g_1g_2k)=1}
\\
\times 
\int_{y_1, y_2} \omega_U(y_1 -y_2) 
\sum_{m > 0} 
\sum_{\chi \shortmod{g_1 g_2 m k}} 
\sum_{\text{(factorizations)}} 
H(\chi, g_1, g_2, k, m)
\cdot 
| \mathcal{C}_1|^2 dy_1 dy_2.
\end{multline*}
A priori, the term $|\mathcal{C}_1|^2$ should be replaced by $|\mathcal{C}_1|^2+|\mathcal{C}_2|^2$, but the symmetry between $\mathcal{C}_1$ and $\mathcal{C}_2$ shows that their contributions are the same.

As a small simplification, we argue that it suffices to estimate the terms with $g_{1,2} = g_{2,2} = m_2 = 1$, i.e., where there are no primes of type II. At such a prime, the size of $H$ is $p^{-1}$ times its size at a prime of type I, so there is a relative savings by this factor of $p$. The set of characters of conductor $p^{\mu + 1}$ may be covered by $p$ cosets of the subgroup of characters of modulus $p^{\mu}$.  We can take this sum over coset representatives to the outside, and absorb them into the definition of the coefficient $\beta$.  The net result is that the bound we get from primes of type II is no worse than the bound from primes of type I. 

It is very awkward to write all the conditions that define the types of primes, so we will not write them inside the summation signs, and instead keep it implicit that the subscripts on the variables indicate the types of primes with which they correspond.  

We observe that with the above factorizations, we have
\begin{equation*}
H(\chi, g_1, g_2, k, m)
\ll 
(Qk)^{\varepsilon}
\sqrt{\frac{(g_{1,1}, g_{2,1})}
	{m_1 \mathrm{rad}(m_3 g_3 k)}}
 =  
(Qk)^{\varepsilon}
\sqrt{\frac{(g_{1}, g_{2})}
	{m_1 g_3 \mathrm{rad}(m_3 g_3 k)}}.
\end{equation*}
Applying these bounds and simplifications, and changing variables $y_2 \rightarrow y_1 + v$, we have
\begin{multline*}
    \mathcal{N}_1^{+}(N', Q, k, Y) \ll (QkY)^{\varepsilon}\max_{ \|\beta\| =1} \max_{\substack{1 \ll U\ll Y}}
    \frac{N'}{\sqrt{U}}
\sum_{g_1 \sim g_2}
\sum_{(a,g_1g_2k)=1}\int_{y_1} 
\sum_{m_1 m_3 \asymp \frac{N''}{(g_1, g_2)}} 
\\
\sum_{\substack{\chi_1\chi_3 \shortmod{m_1 m_3  g_3  k }}} 
\sqrt{\frac{(g_{1}, g_{2})}
	{m_1 g_3 \mathrm{rad}(g_3 m_3 k)}}
\int_{\substack{|v|\ll Y^\varepsilon \\ \text{ or }
|v| \asymp U}}
\Big|
\sum_{\substack{ (b_1', g_1 g_2) = 1}}
\beta(a b_1' g_1,  y_1) (b_1')^{-iv} \chi(b_1') \Big|^2 
 dv dy_1.
\end{multline*}
By the large sieve (Lemma \ref{lem:largesieve}) with $d=m_1 m_3 g_3 k$ we obtain
\begin{multline*}
    \mathcal{N}_1^{+}(N', Q, k, Y) \ll (QkY)^{\varepsilon}\max_{ \|\beta\| =1} \max_{\substack{1 \ll U\ll Y}}
    \frac{N'}{\sqrt{U}}
\sum_{g_1 \sim g_2}
\sum_{(a,g_1g_2k)=1}\int_{y_1} 
\sum_{m_1 m_3 \asymp \frac{N''}{(g_1, g_2)}} 
\\
\sqrt{\frac{(g_{1}, g_{2})}
	{m_1 g_3 \mathrm{rad}(g_3 m_3 k)}}
\Big(m_1 m_3 g_3 k U 
+ \frac{Q}{a g_1}
\Big)
\sum_{b_1'} |\beta(a b_1' g_1, y_1)|^2 
 dy_1.
\end{multline*}
Estimating the $m_1,m_3$-sum using Lemma \ref{lemma:radicalsum}, we obtain
\begin{multline*}
    \mathcal{N}_1^{+}(N', Q, k, Y) \ll (QkY)^{\varepsilon}
   \max_{\substack{1 \ll U\ll Y}}
    \frac{N'}{\sqrt{U}}
\sum_{g_1 \sim g_2}
\sum_{(a,g_1g_2k)=1}
\\
\mathrm{rad}(k)^{-1/2}\Big(kU(N'')^{3/2} 
+ \frac{Q\sqrt{N''}}{a}
\Big)
\int_{y_1} 
\sum_{b_1'} |\beta(a b_1' g_1, y_1)|^2 
 dy_1.
\end{multline*}
Finally we turn to $g_2$. Given $g_1$, the number of $g_2$ is $O(Q^{\varepsilon})$, uniformly in $g_1$.  We can then absorb the sums over $g_1$ and $a$, as well as the integral over $y_1$, into $\beta$, which is assumed to have norm $1$.  Hence
\begin{equation*}
    \mathcal{N}_1^{+}(N', Q, k, Y) \ll (QkY)^{\varepsilon} \max_{U\ll Y} \frac{N'}{\sqrt{\mathrm{rad}(k)U}} (kU (N'')^{3/2} + Q \sqrt{N''}).
\end{equation*}
Inserting $N''=\frac{Q^2}{N'}(Y^{\varepsilon}+U)$, we conclude that
\begin{equation*}
    \mathcal{N}_1^{+}(N', Q, k, Y) \ll T^{\varepsilon} \left(\frac{kQ^3Y^2}{\sqrt{\mathrm{rad}(k)N'}} + Q^2 \sqrt{\frac{N'}{\mathrm{rad}(k)}}\right).
\end{equation*}
The same bound holds for $\mathcal{N}_1^{-}(N', Q, k, Y)$. 
Using Lemma \ref{lemma:Diagonal}, we then have 
\begin{equation}
\label{eq:N1boundFinal}
    \mathcal{N}_1(N', Q, k, Y)\ll T^{\varepsilon} \left(QN'+\frac{kQ^3Y^2}{\sqrt{\mathrm{rad}(k)N'}} + Q^2 \sqrt{\frac{N'}{\mathrm{rad}(k)}}\right),
\end{equation}
Applying \eqref{eq:N1boundFinal} into \eqref{eq:N2normdef} leads to the same bound for $\mathcal{N}$ as holds for $\mathcal{N}_1$.
The same analysis and bound (with $Y=1$) hold for $\mathcal{N}(N',Q,k)$ and 
thus concludes the proof of Theorem \ref{thm:dualnormbound}.

\section{Proof of Theorem \ref{thm:shiftedSum}}\label{sect:ShiftedSum}

While the proof of Theorem \ref{thm:shiftedSum} almost follows directly from the proof of Theorem \ref{thm:mainthm}, there are some technical differences. In this section, we highlight the adjustments needed to complete the proof of Theorem \ref{thm:shiftedSum}.
Let
\begin{align}
    Z=\sum_{n,k}\lambda_f(n)\lambda_f(n+k)W\left(\frac{n}{N},\frac{k}{H}\right).
\end{align}
Applying the delta method to detect $m=n+k$ as in Section \ref{subsect:ApplyingDelta} with $H\leq C\leq N^{1/2-\varepsilon}$, we arrive at 
\begin{equation}
 %\label{eq:I00def}
  Z= 
  \sum_{i=1}^{2}
  \frac{1}{C}
 \sum_{cd \leq C} \frac{1}{cd} F_i\left(\frac{cd}{C} \right)
 \thinspace \thinspace
 \sumstar_{h \shortmod{c}} \mathcal{S},
 \end{equation}
with
\begin{multline}
%\label{eq:Sdef}
 \mathcal{S} = \mathcal{S}_i = 
 \sum_{m,n,k} \lambda_f(m) e_c(-hm) \lambda_f(n) e_c(hn) e_c(hk)
 \\
 \times 
 w_N(m) W\left(\frac{n}{N},\frac{k}{H}\right)
  U_i \left(
 \frac{n+k-m}{cdC}\right).
 \end{multline}
We again replace $U_i$ by $U$ as all the bounds will be independent of $i$.

We now proceed to apply dual summations to $m,n,k$ as in Section \ref{sect:DualSummations}, with the only difference being the weight function $W(n/N,k/H)$ behaves less nicely in terms of $k$ compared to $\widehat{\omega}(Tk/n)$. Indeed, the Poisson summation in $k$ yields 
\begin{multline*}
    \sum_ke_c(hk)W\left(\frac{n}{N},\frac{k}{H}\right)
  U \left(
 \frac{n+k-m}{cdC}\right)
 \\
 =\sum_{k'\equiv h\shortmod{c}}\intR W\left(\frac{n}{N},\frac{z}{H}\right)
  U \left(
 \frac{n-m+z}{cdC}\right)e\left(\frac{k'z}{c}\right)dz.
\end{multline*}
With the choice of $C\geq H$, repeated integration by parts 
shows the integral is very small
unless $|k'|\ll cN^\varepsilon/H$. With this deviation highlighted, we follow the rest of the proof of Proposition \ref{prop:SafterVoronoi} to arrive at\footnote{Except the last change of variable has to be done with $x=Nt^3$ instead of $x=\nu t$.} 
\begin{equation*}
    Z \ll N^{\varepsilon}
\max_{\substack{Q \ll C \\ \eta \in \{ \pm 1 \} }}
\sum_{d \ll C/Q}
\max_{\substack{N' \ll \frac{N^{2+\varepsilon}}{d^3 C^3} }}
HN\left(\frac{N'^2}{NQ^3}\right)^{2/3}  
\int_{t \asymp 1}   \sum_{c \asymp Q} 
 \sum_{|k'|\ll \frac{QN^\varepsilon}{H}} 
|\mathcal{A} |^2
 + O(N^{-100}),
\end{equation*}
where $w(\cdot)$ is $T^\varepsilon$-inert and
\begin{equation*}
%\label{eq:A1def}
\mathcal{A} = 
\mathop{\sum_{n_1 | c} \sum_{n_2}}_{n_1^2 n_2 \asymp N'} 
\frac{\lambda_f(n_1, n_2)}{n_1 n_2}
S(\eta   \overline{k'}, n_2 ;c/n_1) 
e_c\Big(3   (n_1^2 n_2 N )^{1/3} t  \Big)
w(\cdot).
\end{equation*}

We first treat the case $k'=0$. Since $k'\equiv h\quad\shortmod{c}$ and $(h,c)=1$, $k'=0$ forces $c=1$ (and hence $n_1 = 1$). Opening the square,
%in the form 
%$|\sum_{n_2} (\cdots)|^2=\sum_{m_2} (\cdots) \sum_{n_2}\overline{(\cdots)}$, 
the $t$-integral evaluates as 
\begin{align*}
    \int_{t\asymp 1} \omega(t, \cdot)e(((m_2-n_2)N)^{1/3}t)dt=\hat{\omega}(((m_2-n_2)N)^{1/3},\cdot).
\end{align*}
This restricts $|m_2-n_2|\ll N'^{2/3+\varepsilon}/N^{1/3}$ up to a small error. Using the AM-GM inequality with Lemma \ref{lemma:RankinSelbergBound}, we see that the contribution of $k'=0$ to $Z$ is bounded by 
\begin{align*}
    \ll N^{\varepsilon}
\max_{\substack{Q \ll C \\ \eta \in \{ \pm 1 \} }}
\sum_{d \ll C/Q}
\max_{\substack{N' \ll \frac{N^2}{d^3 C^3} T^{\varepsilon}}}
HN\left(\frac{N'^2}{N}\right)^{2/3}\mathop{\sum\sum}_{\substack{m_2,n_2\asymp N'\\|m_2-n_2|\ll N'^{2/3+\varepsilon}/N^{1/3}}}\frac{|\lambda_f(1,n_2)|^2}{n_2^2}\ll \frac{HN^{2+\varepsilon}}{C^3}.
\end{align*}
By symmetry, we may assume $k'>0$.
This leads to 
\begin{equation*}
    Z \ll N^{\varepsilon}
\max_{\substack{Q \ll C \\ \eta \in \{ \pm 1 \} }}
\sum_{d \ll C/Q}
\max_{\substack{N' \ll \frac{N^{2+\varepsilon}}{d^3 C^3} }}
HN\left(\frac{N'^2}{NQ^3}\right)^{2/3}  
\int_{t \asymp 1}   \sum_{c \asymp Q} 
 \sum_{k'\ll \frac{QN^\varepsilon}{H}} 
|\mathcal{A} |^2
 + O\left(\frac{HN^{2+\varepsilon}}{C^3}\right).
\end{equation*}
Following the rest of Section \ref{sect:DualSummations} leads directly to
\begin{multline*}
%\label{eq:secondmomentVSnorm}
Z \ll 
\frac{H N^{2+\varepsilon}}{C^3}
+
HN^\varepsilon
\max_{\eta_4 \in \{\pm 1\}}
\sum_{d}
\max_{\substack{Q \ll C/d \\ N' \ll \frac{N^{2+\varepsilon}}{d^3 C^3}}}
\sum_{k'\ll \frac{QN^\varepsilon}{H}}\Big(
\frac{1}{Q}
\mathcal{N}(N', Q, k', \Phi(N'))
\\
+\frac{1}{Q} \delta(\Phi(N') \ll N^\varepsilon)\mathcal{N}(N',Q,k')
\Big)
,
\end{multline*}
with $\Phi(N')=\frac{(N'N)^{1/3}}{Q}$. Applying Theorem \ref{thm:dualnormbound} %on the norms 
yields 
\begin{multline*}
Z \ll 
\frac{H N^{2+\varepsilon}}{C^3}
+
HN^\varepsilon
\max_{\eta \in \{\pm 1\}}
\sum_d
\max_{\substack{Q \ll C/d \\ N' \ll \frac{N^{2+\varepsilon}}{d^3 C^3}}}\sum_{k'\ll\frac{QN^\varepsilon}{H}} \left(N' + \frac{k'}{\sqrt{\mathrm{rad}(k)}}\frac{ Q^2 \Phi^2}{\sqrt{N'}} + Q \sqrt{\frac{N'}{\mathrm{rad}(k)}}\right)\\
\ll N^\varepsilon\left(\frac{N^2}{C^2}+\frac{CN}{\sqrt{H}}+\sqrt{H}N\right),
\end{multline*}
which holds for any $H\leq C\leq N^{1/2-\varepsilon}$.  We will choose $C=N^{1/3}H^{1/6}+H$ to conclude Theorem \ref{thm:shiftedSum}. 
Note that the assumption $H^2 \leq N^{1-\varepsilon}$ ensures that $N^{1/3} H^{1/6} \leq N^{1/2-\varepsilon}$, so this choice of $C$ is valid.

\section{Proof sketches of the corollaries}\label{sect:CorProofSketch}
\begin{proof}[Proof of Corollary \ref{coro:RankinSelberg} (Rankin-Selberg Problem)]
It suffices to show the bound for the dyadic segment $x/2 < n \leq x$.  Let $w = w_{+}$ be a smooth nonnegative function that is identically $1$ on $[x/2, x]$, and supported on $[x/2-y, x+y]$, with $y = o(x)$ to be chosen later. We can also choose a smooth test function $w_{-}$ that is supported on $[x/2,x]$ which is identically $1$ on $[x/2+y, x-y]$.  Then the characteristic function of $[x/2,x]$ is squeezed between $w_{-}$ and $w_{+}$.  We focus on $w = w_{+}$, but the same ideas work for $w_{-}$ as well.  It is easy to check that 
for $j \geq 1$ that $w^{(j)}(t) \ll y^{-j}$.  
By repeated integration by parts we deduce that the Mellin transform of $w$ satisfies the bound
\begin{equation*}
    |\widetilde{w}(s)| \ll_A \frac{y}{x} x^{\sigma} \Big(1 + \frac{|s| y}{x} \Big)^{-A},
\end{equation*}
for arbitrary $A > 0$, and
assuming say $1/2 \leq \sigma \leq 3/4$. 
By standard Mellin inversion, 
\begin{equation*}
\sum_{x/2 < n \leq x} |\lambda_g(n)|^2
\leq \sum_{n=1}^{\infty} |\lambda_g(n)|^2 w_{+}(n)
= \frac{1}{2 \pi i} \int_{(2)} \widetilde{w}(s) \frac{L(g \otimes g, s)}{\zeta(2s)} ds.
\end{equation*}
Here $L(g \otimes g, s) = L(\mathrm{sym}^2 g, s) \zeta(s)$. We shift contours to the line $1/2+\varepsilon$, crossing a pole at $s=1$ giving rise to the residue $\widetilde{w}(1) \frac{L(\mathrm{sym}^2 g, 1)}{\zeta(2)}$.  Note $\widetilde{w}(1) = x + O(y)$.  On the $1/2+\varepsilon$ line we obtain a bound of the form
\begin{equation}
\label{eq:RankinSelbergIntegralBoundSketch}
\frac{y}{x} x^{1/2+\varepsilon}   \int_{|s| \ll \frac{x^{1+\varepsilon}}{y}} |\zeta(s) L(\mathrm{sym}^2 g, s)| |ds|
\ll \frac{y}{x} x^{1/2+\varepsilon} \Big(\frac{x}{y}\Big)^{1/2} \Big(\Big(\frac{x}{y}\Big)^{4/3}\Big)^{1/2} 
= \frac{x^{2/3+\varepsilon}}{ y^{1/6}},
\end{equation}
using the Cauchy-Schwarz inequality, followed by \eqref{eq:mainthm} and the easily obtained second moment bound for the Riemann zeta function.
The two error terms balance when $y = x^{2/3} y^{-1/6}$, i.e., $y=x^{4/7}$, giving Corollary \ref{coro:RankinSelberg}.
\end{proof}

\begin{proof}[Proof of Corollary \ref{coro:LNQ}]
Lin, Nunes, and Qi \cite[Section 1.5]{LNQ} explicitly point out that any improvement on the second moment leads to a better bound in their problem.  It is a simple matter of bookkeeping to see how Theorem \ref{thm:mainthm} translates into their work.
\end{proof}

\begin{proof}[Proof of Corollary \ref{coro:nonlinear}]
The idea here is very similar to that in the Rankin-Selberg problem.  By Mellin inversion and stationary phase, we have that
\begin{equation*}
    w(n) e(P (n/N)^{\beta})
    \sim P^{-1/2} \int_{|\mathrm{Im}(s)| \asymp P} F(s) n^{-s} ds,
\end{equation*}
for some function $F(s) \ll N^{\sigma}$ with $\sigma = \mathrm{Re}(s)$, and with a very small error term.  Then the sum at hand equals
\begin{equation*}
    P^{-1/2} \int_{|\mathrm{Im}(s)| \asymp P} F(s) \frac{L(g \otimes g, s)}{\zeta(2s)} ds.
\end{equation*}
We shift the contour to the line $\sigma = 1/2 + \varepsilon$.  The pole at $s=1$ is not encountered here since $F$ is very small at this point.  On the new contour we use the same method as in \eqref{eq:RankinSelbergIntegralBoundSketch}, leading to the bound in Corollary \ref{coro:nonlinear}.  
\end{proof}

\printbibliography

\end{document}